\documentclass[11pt]{amsart}
\usepackage{pgf,tikz}
\usetikzlibrary{arrows}
\usepackage{amsmath}
\usepackage{amssymb}
\usepackage{enumerate}
\usepackage{pdfsync}
\usepackage{amsrefs}

\newcommand{\R}{{\mathbb R}}       

\newcommand{\Z}{{\mathbb Z}}       
\newcommand{\BB}{{\mathcal B}}
\newcommand{\DD}{{\mathcal D}}
\newcommand{\FF}{{\mathcal F}}
\newcommand{\HH}{{\mathcal H}}
\newcommand{\GG}{{\mathcal G}}
\newcommand{\LL}{{\mathcal L}}

\newcommand{\AZ}{{\mathcal A}}

\newcommand{\TT}{{\mathcal T}}

\newcommand{\RR}{{\mathcal R}}
\newcommand{\CH}{{\mathcal Ch}}

\newcommand{\SSS}{{\mathcal S}}

\newcommand{\CC}{{\mathcal C}}

\newcommand{\diam}{\operatorname{diam}}
\newcommand{\dist}{\operatorname{dist}}

\newcommand{\rf}[1]{{\eqref{#1}}}

\newcommand{\supp}{\operatorname{supp}}

\newcommand{\vphi}{{\varphi}}
\newcommand{\ve}{{\varepsilon}}
\newcommand{\vv}{{\vspace{2mm}}}

\newcommand{\wt}[1]{{\widetilde{#1}}}
\newcommand{\wh}[1]{{\widehat{#1}}}

\newcommand{\rest}{{\lfloor}}
\newcommand{\sss}{\operatorname{Stop}}
\newcommand{\ttt}{\operatorname{Top}}
\newcommand{\tree}{\operatorname{Tree}}
\newcommand{\roo}{\operatorname{Root}}
\newtheorem{theorem}{Theorem}[section]
\newtheorem{lemma}[theorem]{Lemma}

\newtheorem{coro}[theorem]{Corollary}

\newtheorem*{theorem*}{Theorem}

\theoremstyle{definition}
\newtheorem{definition}[theorem]{Definition}

\theoremstyle{remark}
\newtheorem{rem}[theorem]{\bf Remark}

\numberwithin{equation}{section}

\newcommand{\brem}{\begin{rem}}
\newcommand{\erem}{\end{rem}}

\definecolor{uuuuuu}{rgb}{0.27,0.27,0.27}

\begin{document}

\title[Riesz transforms on uniformly disconnected sets]{Riesz transforms of non-integer homogeneity on uniformly disconnected sets}

\author{Maria Carmen Reguera and Xavier Tolsa}

\address{Maria Carmen Reguera. Departament de
Ma\-te\-m\`a\-ti\-ques, Universitat Aut\`onoma de Bar\-ce\-lo\-na, Catalonia, Spain and School of Mathematics, University of Birmingham, Birmingham, UK.} \email{m.reguera@bham.ac.uk}

\address{Xavier Tolsa. Instituci\'{o} Catalana de Recerca i Estudis Avan\c{c}ats (ICREA) and Departament de
Ma\-te\-m\`a\-ti\-ques, Universitat Aut\`onoma de Bar\-ce\-lo\-na, Catalonia} \email{xtolsa@mat.uab.cat}

\thanks{Both authors were partially supported by grants 2009SGR-000420 (Catalonia), and MTM-2010-16232 and  MTM2013-44304-P (Spain). M.C.R. was also supported by the Juan de la Cierva programme 2011.
X.T. was also funded by the European Research
Council under the European Union's Seventh Framework Programme (FP7/2007-2013) /
ERC Grant agreement 320501 and by 2014 SGR 75 (Catalonia).}

\begin{abstract}
In this paper we obtain precise estimates for the $L^2$ norm of the $s$-dimensional Riesz transforms
on very general measures supported on Cantor sets in $\R^d$, with $d-1<s<d$. From these estimates we
infer that, for the so called uniformly disconnected compact sets, the capacity $\gamma_s$ associated with the Riesz kernel $x/|x|^{s+1}$ is comparable
to the capacity $\dot{C}_{\frac{2}{3}(d-s),\frac{3}{2}}$ from non-linear potential theory.  
\end{abstract}

\maketitle

\section{Introduction}

In this paper we provide estimates from below for the $L^{2}$ norm of the $s$-dimensional Riesz transform of measures supported on very general Cantor sets in $\R^d$, for $s\in(d-1,d)$. The bounds obtained are written in terms of the densities of the cubes appearing in the construction of the Cantor sets. Our estimates 
 allow us to establish an equivalence between the capacity $\gamma_{s}$ associated with the s-dimensional Riesz kernel and the capacity $\dot{C}_{\frac{2}{3}(d-s),\frac{3}{2}}$ from non-linear potential theory for the so called uniformly disconnected compact sets.  

In this paper we combine some of the techniques of previous works on the 
subject by Mateu and Tolsa \cite{MT}, \cite{T}, with others from Eiderman, Nazarov, Tolsa, and Volberg \cite{ENV12}, \cite{NToV}.
 The general case for arbitrary compact sets still remains open. 
 
 To state our results precisely, we need to introduce some notation.
For $0<s<d$ and $x\in\R^d\setminus\{0\}$, we denote
$$K^s(x)=\frac x{|x|^{s+1}},$$
and we let $\RR^s$ be the associated Riesz transform, so that for a measure $\mu$ in $\R^d$ and $x\in\R^d$,
$$\RR^s \mu(x) =\int K^s(x-y)\,d\mu(y),$$
whenever the integral makes sense. To avoid delicate issues with convergence, we will work with the truncated Riesz transform
$$\RR^s_{\ve}  \mu(x) =\int_{|x-y|>\ve} K^s(x-y)\,d\mu(y).$$
We say that $\RR^s \mu$ is bounded in $L^{2}(\mu)$ if the truncated Riesz transforms $\RR^s_{\ve} \mu$ are bounded in $L^{2}(\mu)$ uniformly in $\ve$.

We now construct the Cantor sets $E$ that we will study, by the following algorithm. Let $Q^0\subset \R^d$ be a compact set. Take now disjoint closed subsets $Q^1_1,\ldots,Q^1_{N_1}\subset Q^0$
and set $E_1= \bigcup_{i=1}^{N_1} Q_i^1$. In the general step $k$ of the construction,
we are given a family of closed sets $Q_1^k,\ldots,Q^k_{N_k}$. Then for each set $Q_i^k$ we take
a finite family of closed sets $Q^{k+1}_j$, $j\in I_{Q_i^k}$, contained in $Q_i^k$ and we denote 
$$\CH(Q_i^k)=\{Q^{k+1}_j\}_{j\in I_{Q_i^k}}$$
(the notation $\CH(Q_i^k)$ stands for ``children'' of $Q_i^k$).
Renumbering these cubes if necessary, we set
$$\{Q_1^{k+1},\ldots,Q_{N_{k+1}}^{k+1}\} := \bigcup_{j=1}^{N_k}\CH(Q_i^k).$$
Then we denote 
$$E_k = \bigcup_{i=1}^{N_{k+1}} Q_i^{k+1},\qquad E=\bigcap_{k=1}^\infty E_k,$$
so that $E$ is a compact set. 
The sets $Q^k_i$ in this construction will be called ``cubes'', although they need not be ``true'' cubes.
The side length of $Q_i^k$ is $\ell(Q_i^k):=\diam(Q_i^k)$.
In this construction we assume that, for all $i,k$,  
\begin{equation}\label{eq000}
\frac18\,\ell(Q_i^k)\leq \ell(Q_j^{k+1})\leq \frac13\,\ell(Q_i^k)\quad\mbox{ for $Q_j^{k+1}\in\CH
(Q_i^k)$,}
\end{equation}
and the separation condition
\begin{equation}\label{eqsepara}
\dist(Q_j^{k+1},Q_{j'}^{k+1})\geq c_{sep}\,\ell(Q_i^k)\quad \mbox{for all $Q_j^{k+1}, Q_{j'}^{k+1}\in\CH
(Q_i^k)$},
\end{equation}
for some fixed constant $c_{sep}>0$. These conditions guaranty that $E$ is a totally disconnected set.
 In particular, we infer that
$$\dist(Q_i^k,\,E\setminus Q_i^k)\geq c^{-1}\,\ell(Q_i^k).$$
Notice that we allow the family of children of $Q_i^k$ to be formed by a single cube $Q_j^{k+1}$.
On the other hand, the assumptions \rf{eq000} and \rf{eqsepara} imply that the number of children is
bounded above uniformly.

\vv
We denote by $\DD$ the family of all the cubes $Q_i^k$ in the construction above. That is,
$$\DD = \{Q_i^k\}_{\substack{1\leq k \leq\infty \\ 1\leq i\leq N_k}}.$$
Given a measure $\mu$ supported on $E$ and a cube $Q\in\DD$, we consider the $s$-dimensional density of 
$\mu$ on $Q$:
$$\Theta^{s}_{\mu}(Q) = \frac{\mu(Q)}{\ell(Q)^s}.$$
We can now state our main result.

\begin{theorem}\label{teopri}
Let $s\in (d-1,d)$.
Let $\mu$ be a finite Borel measure supported on the Cantor set $E\subset\R^d$ described above. Suppose that 
$$\sup_{Q\in \DD} \Theta_\mu^s(Q)\leq C.$$
Then, 
\begin{equation}\label{eqdyad1}
\|\RR^s \mu\|_{L^2(\mu)}^2 \approx  \sum_{Q\in\DD} \Theta_\mu^s(Q)^2\,\mu(Q),
\end{equation}
where the comparability constant depends only on $s$, $d$ and $c_{sep}$.
\end{theorem}

We remark that the estimate from above is already known for any $0<s<d$ due to the work of Eiderman, Nazarov and Volberg \cite{ENV10} (extending previous arguments by Mateu, Prat and Verdera \cite{MPV} for the case $0<s<1$). So the novelty lies in the converse inequality. This is a very delicate estimate, in particular because we are not in the realm where the notion of Menger curvature is useful, namely the case $0<s\leq1$.

In the case $0<s<1$,  Mateu, Prat and Verdera \cite{MPV} have shown that\rf{eqdyad1} holds for any arbitrary compact set $E$ (with $\DD$ being the family of all dyadic cubes). Mateu and Tolsa \cite{MT} have studied the general case $0<s<d$ when $\mu$ is the probability measure on a Cantor set in which the densities of the cubes in the construction decrease with side length. This unnecessary restriction is removed by Tolsa in \cite{T}. The Cantor sets considered in \cite{MT} and \cite{T} are some kind of high-dimensional variants of the well known $1/4$ planar Cantor set.
The advantage of these Cantor sets is that, in a given stage $k$ of the construction, all the cubes $Q_i^k$ have the same side lengths and densities. In this paper, we go a step forward by considering much more general Cantor sets and measures where we no longer have the aforementioned properties. On the down side, we have to restrict ourselves to $d-1<s<d$ due to the use of a maximum principle for the $s$-dimensional Riesz transforms which, apparently, fails for $s<d-1$.

Theorem \ref{teopri} can also be understood as a refined quantitative version of the results of Prat \cite{Pr}, Vihtil\"a \cite{Vih} and Eiderman, Nazarov and Volberg \cite{ENV12}, in the
particular case of the preceding Cantor sets. In these works it is shown that,
given $F\subset\R^d$ with $0<\HH^s(F)<\infty$ and $\mu = \mathcal H^{s}|_F$,  
 the $s$-dimensional Riesz transform with respect to  $\mu$ is unbounded in $L^2(\mu)$, in the cases $0<s<1$ \cite{Pr}; $s\in (0,d)\setminus\Z$, $\mu$ with positive lower $s$-dimensional density \cite{Vih}; and for $d-1<s<d$, $\mu$ with zero lower $s$-dimensional density \cite{ENV12}.  
 Finally, there is another quantitative result worth mentioning in the work of Jaye, Nazarov and Volberg \cite{JNV} that relates the boundedness of the fractional $s$-dimensional Riesz transform $d-1<s<d$ with a weak type estimate for a Wolff potential of exponential type.

Theorem \ref{teopri} has an important collorary regarding the capacities
$\gamma_s$ and $ \dot C_{\frac23(d-s),\frac32}$. To present the corollary we need some extra definitions.
Given a compact set $F\subset \R^d$, the capacity $\gamma_s$
of $F$ is defined by
$$
\gamma_s(F) = \sup|\langle T,1\rangle|,
$$
where the supremum runs over all distributions $T$ supported
on $F$ such that $\|\RR^{s}(T)\|_{L^\infty(\R^d)} \leq 1$.

Denote by $\Sigma(F)$ the family of measures $\mu$ supported on $F$ such that  
$\mu(B(x,r))\leq r^n$ for all $x\in\R^d$ and $r>0$. It turns out that
$$\gamma_s(F) \approx \sup\{\mu(F):\,\mu\in\Sigma(F),\,\|\RR^s\mu\|_{L^2(\mu)}^2\leq \mu(F)\}.$$
This was first shown for $s=1,d=2$ by Tolsa \cite{Tolsa-sem}, and it was extended to
the case $s=d-1$ by Volberg \cite{Vo}, and to the other values of $s$ and $d$ by Prat \cite{Laura}.

Now we turn to non linear potential theory. Given
$\alpha>0$ and $1<p<\infty$ with $0<\alpha p <d$, the capacity $\dot C_{\alpha,p}$
of $F\subset\R^d$ is defined as
$$\dot C_{\alpha,p}(F) = \sup_\mu \mu(F)^p,$$
where the supremum is taken over all positive measures $\mu$ supported on $F$ such that
$$I_\alpha(\mu)(x) = \int \frac1{|x-y|^{d-\alpha}}\,d\mu(x)$$
satisfies $\|I_\alpha(\mu)\|_{p'}\leq 1$, where as usual $p'=p/(p-1)$. 

We are interested in the characterization of $\dot C_{\alpha,p}$ in terms of Wolff
potentials. Consider
$$\dot W^\mu_{\alpha,p}(x) = \int_0^\infty \biggl(\frac{\mu(B(x,r))}{r^{d-\alpha p}}\biggr)^{p'-1}\,\frac{dr}r.$$
A classical theorem of Wolff establishes
\begin{equation*}
\dot C_{\alpha,p}(F) \approx \sup_\mu \mu(F),
\end{equation*}
where the supremum is taken over all measures $\mu$ supported on $F$ such that
$\int \dot W_{\alpha,p}^\mu(x)\,d\mu\leq \mu(F)$ (see \cite[Chapter 4]{Adams-Hedberg},
for instance).

Finally we wish to remark that the class of Cantor sets $E$ considered in Theorem \ref{teopri}
coincides with the class of compact {\em uniformly disconnected sets}. According to 
\cite[p.156]{David-Semmes-fractured}, a set $F\subset\R^d$ is called uniformly disconnected
if there exists a constant $c_F>0$ so that for each $x\in F$ and $r>0$ one can find a
closed  (with respect to $F$) subset $A\subset F$ such that $A\subset B(x,r)$, $A\supset F\cap B(x,c_F^{-1}r)$,
and $\dist(A,F\setminus A)\geq c_F^{-1}r$. One can check that any uniformly disconnected set can be constructed as one of the Cantor sets $E$ considered in Theorem \ref{teopri}, 
for a suitable separation constant $c_{sep}$, and replacing the constants $1/8$ and $1/3$ in \rf{eq000}
by others if necessary. Conversely, it is immediate  that any such
set $E$ is uniformly disconnected.

We are now ready to formulate the corollary to Theorem \ref{teopri}.

\begin{coro}
Let $F\subset\R^d$ be compact and uniformly disconnected, with constant $c_F$, and let $d-1<s<d$. Then
\begin{equation} \label{e.capacities}
\gamma_s(F) \approx \dot C_{\frac23(d-s),\frac32}(F),
\end{equation}
with the comparability constant depending only on $d$, $s$, and $c_F$.
\end{coro}

The estimate \eqref{e.capacities} was proved by Mateu, Prat and Verdera when $0<s<1$ in \cite{MPV}. 
By using the upper estimate in the inequality \rf{eqdyad1},  Eiderman, Nazarov and Volberg \cite{ENV10}, 
showed that 
for all indices $0<s<d$
$$\gamma_s(F) \gtrsim \dot C_{\frac23(d-s),\frac32}(F),$$ 
for any compact set $F\subset\R^d$.
It is known that the opposite inequality is false when $s$ is integer, see \cite{ENV10}. 
On the other hand, in the works of Tolsa and Mateu \cite{MT} and \cite{T} mentioned above, it is shown that the comparability \eqref{e.capacities} holds for the Cantor sets studied in these papers for $0<s<d$.
Although our corollary extends the result to a more general family of compact sets when $d-1<s<d$, the general case when $0<s<d$ is non integer and $F$ is a general compact set remains open. 

From Theorem \ref{teopri} it follows easily that 
the comparability
$\gamma_s(E) \approx \dot C_{\frac23(d-s),\frac32}(E)$ holds for the Cantor sets defined above, with the constant in the comparability depending only on $d$, $s,$ and the separation constant $c_{sep}$. Indeed, recall that one just has to show that $\gamma_s(E) \lesssim \dot C_{\frac23(d-s),\frac32}(E)$. To this end,
take $\mu\in\Sigma(E)$ with $\|R^s\mu\|_{L^2(\mu)}^2\leq \mu(E)$ such that
$\gamma_s(E) \approx \mu(E)$. Then, by Theorem \ref{teopri},
$$\sum_{Q\in\DD} \Theta_\mu^s(Q)^2\,\mu(Q)\approx \|\RR^s \mu\|_{L^2(\mu)}^2 \leq\mu(E).$$
It is easy to check that the above sum on the left had side is comparable to
$\int \dot W_{\frac23(d-s),\frac32}^\mu(x)\,d\mu$. So one infers that
$\gamma_s(E) \lesssim \dot C_{\frac23(d-s),\frac32}(E)$, as wished.

In order to prove Theorem \ref{teopri} we will consider a stopping time argument. The stopping conditions will take into account the oscillations of the densities on the different cubes from $\DD$ and  
 the possible large values of the $s$-dimensional Riesz transform on each cube. 
 In this way we will split $\DD$ into different families of  cubes, which we will call ``trees'',
 following an approach similar to the one in \cite{T}. One the main differences of the present work with respect to the latter reference is that, in some key steps,  our work paper implements a variational argument 
 borrowed from work of Eiderman, Nazarov, Volberg \cite{ENV12} and Nazarov, Tolsa and Volberg \cite{NToV}, suitably adapted to our setting. This variational argument requires the
 $s$-dimensional Riesz transforms to satisfy the maximum principle mentioned above.

The paper is organized as follows. Section 2 contains the basic background. Section 3 is devoted to the description of the stopping time argument and the properties associated to it. In Section 4 we prove that the sizes of trees obtained by the stopping time argument must be small. We will use a touching point argument for this purpose, where fact that $s$ is fractional will play an important role. In Section 5 we describe four relevant families of enlarged trees and we start analysing the easier ones. Section 6 contains some Fourier analysis that will be necessary for the development of Section 7. The analysis of the most difficult family of trees is included in Section 7. And at last, Section 8 puts together all the estimates obtained in previous sections to provide the proof of the Main Theorem \ref{teopri}.


\section{Preliminaries}

\subsection{About cubes, trees}

Below, to simplify notation we will write $\RR$, $K$ and $\Theta$ instead of $\RR^s$, $K^s$ and $\Theta_\mu^s$, respectively.

Notice that if we replace the cubes $Q\in\DD$ by 
$$\wh Q = \Bigl\{x\in \R^d: \dist(x,Q)\leq \frac1{10}c_{sep}\,\ell(Q)\Bigr\},$$
the separation condition \rf{eqsepara} still holds with a slightly worse constant. 
Analogously, \rf{eq000} is still satisfied, possibly after modifying suitably the constants $1/8$ and $1/3$. The advantage of $\wh Q$ over $Q$ is that the Lebesgue measure of $\wh{Q}$ is comparable to 
$\diam(\wh Q)^d$, which is not 
guaranteed for $Q$. To avoid some technicalities, in the proof of Theorem \ref{teopri} we will assume that the original cubes $Q\in\DD$
satisfy $\LL^d(Q)\approx\ell(Q)^d$, where $\LL^d$ stands for the Lebesgue measure on $\R^d$.
Moreover, we will assume that the separation constant $c_{sep}$ does not exceed $1/10$, say.

For $j\geq0$, we denote by $\DD_j$ the family of cubes of generation $j$ that appear in the construction of $E$, and we set $\DD=\bigcup_{j\geq0} \DD_j$. We assume that $\mu(Q)>0$ for all $Q\in\DD$. Otherwise we
eliminate $Q$ from the construction of $E$.
If $R\in\DD_j$, we denote by $\DD_k(R)$ the family of the cubes from $\DD_{j+k}$ which are contained in $R$.
Notice that if $Q\in\DD_k(R)$, then
$$8^{-k}\leq\frac{\ell(Q)}{\ell(R)}\leq 3^{-k}.$$

Given a cube $Q\in\DD$, for any constant $a>1$ we denote
$$aQ = \{x\in\R^d:\dist(x,Q)\leq (a-1)\,\ell(Q)\}.$$
Also, we set
$$p(Q) =\sum_{P\in\DD:P\supset Q} \frac{\ell(Q)}{\ell(P)}\,\Theta(P).$$
So $p(Q)$ should be understood as a smoothened version of $\Theta(Q)$. Also, given $Q,R\in\DD$ with 
$Q\subset R$, we set
$$p(Q,R) =\sum_{P\in\DD:Q\subset P\subset R} \frac{\ell(Q)}{\ell(P)}\,\Theta(P).$$

We say that a cube $Q$ is  $p$-doubling (with constant $c_{db}$) if 
\begin{equation}\label{eqdob}
p(Q)\leq c_{db}\,\Theta(Q).
\end{equation}

Given a family of cubes $\TT\subset \DD$, we denote
$$\sigma(\TT) =\sum_{Q\in\TT} \Theta(Q)^2\,\mu(Q).$$
So Theorem \ref{teopri} asserts that $\|\RR \mu\|_{L^2(\mu)}^2 \approx\sigma(\DD)$ under the assumption $\sup_{Q\in \DD} \Theta_\mu^s(Q)\leq C$.
\vv

\begin{lemma}\label{lemdob0}
Suppose that $c_{db}$ is big enough, depending on $a,s,d$.
Let $Q_0,Q_1,\ldots,Q_n$ be a family of cubes from $\DD$ such that $Q_j$ is son of $Q_{j-1}$ for $1\leq j\leq 
n$. Suppose that $Q_j$ is not $p$-doubling for $1\leq j\leq n$.
Then
\begin{equation}\label{eqcad35}
\Theta(Q_j)\leq 2^{-j/2}\,p(Q_0).
\end{equation}
\end{lemma}

\begin{proof}
For $1\leq j \leq n$,
the fact that $Q_j$ is not $p$-doubling implies that
\begin{equation}\label{eqsak33}
\Theta(Q_j) \leq \frac1{c_{db}}\,p(Q_j) = \frac1{c_{db}}\Biggl (\sum_{k=0}^{j-1} \frac{\ell(Q_j)}{\ell(Q_{j-k})}\,
\Theta(Q_{j-k})+ \frac{\ell(Q_j)}{\ell(Q_0)}\,p(Q_0)\Biggr).
\end{equation}
We will prove \rf{eqcad35} by induction on $j$. For $j=0$ this is in an immediate consequence of the
definition of $p(Q)$. Suppose that \rf{eqcad35} holds for $0\leq h\leq j$, with $j\leq n-1$, and let us 
consider the case $j+1$. From \rf{eqsak33} and the induction hypothesis we get
\begin{align*}
\Theta(Q_{j+1}) & \leq  \frac1{c_{db}}\Biggl (\Theta(Q_{j+1}) + \sum_{k=1}^j \frac{\ell(Q_{j+1})}{\ell(Q_{j+1-k})}\,
\Theta(Q_{j+1-k})+ \frac{\ell(Q_{j+1})}{\ell(Q_0)}\,p(Q_0)\Biggr)\\
&\leq \frac1{c_{db}}\Biggl (\Theta(Q_{j+1}) + \sum_{k=1}^j 3^{-k}\,
\Theta(Q_{j+1-k})+ 3^{-j-1}\,p(Q_0)\Biggr)\\
&\leq \frac1{c_{db}}\Biggl (\Theta(Q_{j+1}) + \sum_{k=1}^j 3^{-k}\,2^{(-j-1+k)/2}p(Q_0)
+ 3^{-j-1}\,p(Q_0)\Biggr)
\end{align*}
Since $\sum_{k=1}^j 3^{-k}\,2^{(-j-1+k)/2}\lesssim 2^{-j/2}$, we obtain
\begin{align*}
\Theta(Q_{j+1})  &\leq  \frac1{c_{db}}\bigl (\Theta(Q_{j+1}) + C\,2^{-j/2}\,p(Q_0)+ 3^{-j-1}\,p(Q_0)\bigr)\\
&\leq \frac1{c_{db}}\bigl (\Theta(Q_{j+1}) + \tilde{C}\,2^{-j/2}\,p(Q_0)\bigr) \\
\end{align*}
It is straightforward to check that yields $\Theta(Q_{j+1})\leq 2^{-(j+1)/2}\,p(Q_0)$ if $c_{db}$
is big enough.
\end{proof}

\vv

For the rest of the paper we will fix $c_{db}$ so that \rf{eqcad35} holds.

\begin{lemma}\label{lemdob}
For a fixed $Q\in\DD$, let $J\subset \DD$ be a family of cubes contained in $Q$ such that for every $P\in J$, every 
cube $P'\in\DD$ such that $P\subset P'\subset Q$ is not $p$-doubling. Then
$$\sigma(J)\leq 2\,p(Q)^2\,\mu(Q).$$
\end{lemma}

\begin{proof}
For $j\geq0$, let $J_j$ be the subfamily of the cubes from $J$ which are $j$ generations below $Q$. That
is, $P$ is from $J_j$ if it belongs to $J$ and it is an $j$-th descendant of $Q$. By the preceding lemma, 
it turns out that 
$$\Theta(P)\leq 2^{-j/2}\,p(Q)\qquad\mbox{ if $P\in J_j$.}$$
Taking also into account that the cubes from $J_j$ are pairwise disjoint, we get
$$\sigma(J_j)\leq 2^{-j}\,p(Q)^2\sum_{P\in J_j:\subset Q}\mu(P) \leq 2^{-j}\,p(Q)^2\,\mu(Q).$$
Therefore,
$$\sigma(J) = \sum_{j=0}^n \sigma(J_j) \leq \sum_{j=0}^\infty 2^{-j}\,p(Q)^2\,\mu(Q) = 2\,p(Q)^2\,\mu(Q),$$
where $n$ is the maximum number of generations between $Q$ and the cubes belonging to $J$.
\end{proof}


\vv
\subsection{The operators $D_Q$}

Given a cube $Q\in\DD$ and a function $f\in L^1(\mu)$, we denote by $m_Qf$ the mean of $f$ on $Q$ with respect to $\mu$. That is, $m_Qf =\frac1{\mu(Q)}\int_{Q} f\,d\mu$. Then we define
$$D_Q f = \sum_{P\in\CH(Q)} \chi_P\,(m_Pf - m_Q f).$$ 
The functions $D_Qf$, $Q\in\DD$, are orthogonal, and it is well known that
$$\|f\|_{L^2(\mu)}^2 = \sum_{Q\in\DD}\|D_Q f\|_{L^2(\mu)}^2.$$

\vv
\subsection{About the Riesz transform}

In the following lemma we collect a pair of useful estimates about the Riesz transform.

\begin{lemma}\label{lemcomp}
Let $Q,R\in\DD$ with $Q\subset R$, and $x,y\in Q$. Then
\begin{equation}\label{eqfacc00}
|\RR(\chi_{R\setminus Q}\mu)(x)|\lesssim \sum_{P\in\DD:Q\subset P\subset R} \frac{\mu(P)}{\ell(P)^s}
\end{equation}
and
\begin{equation}\label{eqfacc11}
|\RR(\chi_{R\setminus Q}\mu)(x)- \RR(\chi_{R\setminus Q}\mu)(y)| \lesssim \frac{|x-y|}{\ell(Q)}\,p(Q,R).
\end{equation}
Also,
\begin{equation}\label{eqfacc12}
\bigl|\RR(\chi_{R\setminus Q}\mu)(x)- \bigl (m_Q(\RR\mu) - m_R(\RR\mu)\bigr)\bigr|\lesssim p(Q,R)  + 
p(R).
\end{equation}
\end{lemma}

\begin{proof}
The first and second inequalities follow by very standard methods, taking into account the separation property \rf{eqsepara}. Regarding \rf{eqfacc12}, by the antisymmetry of the Riesz transform, we have\begin{align*}
m_Q(\RR\mu) - m_R(\RR\mu) & = m_Q(\RR(\chi_{Q^c}\mu)) - m_R(\RR(\chi_{R^c}\mu))\\
& =  m_Q(\RR(\chi_{R\setminus Q})) + m_Q(\RR(\chi_{R^c}\mu))  - m_R(\RR(\chi_{R^c}\mu)).
\end{align*}
From the estimate \rf{eqfacc11}, we deduce that for all $x'\in Q\subset R$, $y'\in R$,
$$|\RR(\chi_{R^c}\mu)(x')- \RR(\chi_{R^c}\mu)(y')| \lesssim p(R).$$
Averaging, we get
$$|m_Q(\RR(\chi_{R^c}\mu))  - m_R(\RR(\chi_{R^c}\mu))|\lesssim p(R).$$
Analogously, for $x\in Q$, we have
$$\bigl|\RR(\chi_{R\setminus Q})(x) - m_Q(\RR(\chi_{R\setminus Q})\bigr|\lesssim p(Q,R).$$
Therefore,
\begin{align*}
\bigl|\RR(\chi_{R\setminus Q})(x)- \bigl (m_Q(R\mu) - m_R(R\mu)\bigr)\bigr| & \leq
\bigl|\RR(\chi_{R\setminus Q})(x)- m_Q(\RR(\chi_{R\setminus Q}))\bigr| \\&\quad +
\bigl|m_Q(\RR(\chi_{R^c}\mu))  - m_R(\RR(\chi_{R^c}\mu))\bigr|\\
& \lesssim p(Q,R)  + p(R).
\end{align*}
\end{proof}
\vv


\subsection{The  operators}

Let $\Phi:\R^d\to[0,\infty)$ be a $1$-Lipschitz function.
Below we will need to work with the suppressed kernel
\begin{equation}\label{eqsuppressed}
K_\Phi(x,y) = \frac{x-y}{\bigl (|x-y|^2+\Phi(x)\Phi(y)\bigr)^{(s+1)/2}}
\end{equation}
and the associated operator 
$$\RR_\Phi\nu(x) =\int K_\Phi(x,y)\,d\nu(y),$$
for a signed measure $\nu$ in $\R^d$.
This kernel (or a variant of this) appeared for the first
time in the work of Nazarov, Treil and Volberg in connection with Vitushkin's conjecture (see
\cite{Vo}). 
For $f\in L^1_{loc}(\mu)$, one denotes $\RR_{\Phi,\mu} f = \RR_\Phi (f\,\mu)$.

If $\ve\approx\Phi(x)$, then it follows that, for any signed measure $\nu$ in $\R^d$,
\begin{equation}
\label{e.compsup''}
\bigl|\RR_{\ve}\nu(x) - \RR_{\Phi}\nu(x)\bigr|\lesssim  \sup_{r> \Phi(x)}\frac{|\nu|(B(x,r))}{r^s}.
\end{equation}
See Lemmas 5.4 and 5.5 in \cite{To}, for example. 

The following result is an easy consequence of a $Tb$ theorem of Nazarov, Treil and Volberg.
See Chapter 5 of \cite{To}, for example.

\begin{theorem}\label{teontv}
Let $\mu$ be a Radon measure in $\R^d$ and let $\Phi:\R^d\to[0,\infty)$ be a $1$-Lipschitz function. Suppose that
\begin{itemize}
\item[(a)] $\mu(B(x,r))\leq c_0\,r^s$ for all $r\geq \Phi(x)$, and
\item[(b)] $\sup_{\ve>\Phi(x)}|\RR_\ve\mu(x)|\leq c_1$.
\end{itemize}
Then $\RR_{\Phi,\mu}$ is bounded in $L^p(\mu)$, for $1<p<\infty$, with a bound on its norm depending only on $p$, $c_0$ and
$c_1$. In particular, $\RR_\mu$ is bounded in $L^p(\mu)$ on the set $\{x:\Phi(x)=0\}$.
\end{theorem}


\section{The corona decomposition}\label{seccorona}

Recall that the starting cube  from the construction of $E$ is denoted by $Q^0$.
Below we choose constants $B\gg M\gg1\gg \delta_0$. For convenience we assume $B$ to be a power of $2$.  

Given a cube $Q$, we define its dyadic density $\Theta_{d}(Q)$,
\begin{equation*}
\Theta_{d}(Q):= 2^{j},
\end{equation*}
where $j\in\Z$ is such that $ 2^{j}\leq \Theta(Q) < 2^{j+1}$.

Given a cube $R\in\DD$, we define families $HD_0(R)$, $LD_0(R)$ and $BR_0(R)$ of cubes from $\DD$ as follows:
\begin{itemize}
\item We say that a cube $Q\subset\DD$  belongs to $HD_0(R)$ if it is contained in $R$, $\Theta_{d}(Q)=B\Theta_{d}(R)$ and has maximal side length. Recall that $B$ is a power of $2$.
\item A cube $Q\subset\DD$  belongs to $LD_0(R)$ if it is contained in $R$, $\Theta(Q)\leq \delta_0\,\Theta(R)$,
it is not contained in any cube from $HD_0(R)$, and has maximal side length.
\item A cube $Q\subset\DD$  belongs to $BR_0(R)$ if it is contained in $R$, it satisfies
$$|m_Q \RR\mu - m_R \RR\mu|\geq M\,\bigr(\Theta(R) + p(Q)\bigl),$$
it is not contained in any cube from $HD_0(R)\cup LD_0(R)$, and has maximal side length.
\end{itemize}

We consider a ``doubling constant'' $c_{db}>10$ in \rf{eqdob}. Then we denote by $HD(R)$, $LD(R)$ and $BR(R)$
the families of maximal, and thus disjoint, $p$-doubling cubes (with constant $c_{db}$) which are contained in $HD_0(R)$, $LD_0(R)$ and $BR_0(R)$, respectively.
We denote 
$$\sss_0(R)= HD_0(R)\cup LD_0(R)\cup BR_0(R)$$
and
$$\sss(R)= HD(R)\cup LD(R)\cup BR(R).$$
For $Q\in\sss_0(R)$, 
let $J_Q$ be the family of cubes from $\DD$ which are contained in $Q$ and are not contained in any cube from $\sss(R)$. Notice that, by Lemma \ref{lemdob},
\begin{equation}\label{eqsigma1}
\sigma(J_Q)\leq c\,p(Q)^2\,\mu(Q).
\end{equation}

For $k\geq 1$, we define $\sss^k(R)$ inductively: we set $\sss^1(R)=\sss(R)$, and for $k>1$
a cube belongs to $\sss^k(R)$ if it belongs to $\sss(Q)$ for some $Q\in\sss^{k-1}(R)$. 
Now we construct the family $\ttt\subset\DD$ as follows:
$$\ttt = \{Q_0\} \bigcup_{k\geq1} \sss^k(Q_0).$$

Given a cube $R\in\ttt$, we denote by $\tree(R)$ (and $\tree_0(R)$) the family of cubes contained in $R$ which are not contained in
any cube from $\sss(R)$ (and $\sss_0(R)$, respectively).  
Observe that the cubes from $\sss(R)$ do not belong to $\tree(R)$.
Notice also that 
$$\DD = \bigcup_{R\in\ttt} \tree(R).$$
Moreover, the union is disjoint.

The following lemma is an easy consequence of our construction.

\begin{lemma}\label{lemfac1}
For every $R\in\ttt$, we have:
\begin{itemize}
\item[(a)] $R$ is $p$-doubling (with constant $c_{db}$).
\item[(b)] Every cube $Q\in\tree(R)$ satisfies
$\Theta(Q)\leq 2B\,\Theta(R)$ and $p(Q)\leq (2B + c_{db})\Theta(R)$.

\item[(c)] Every cube $Q\in\tree_0(R)$ satisfies
\begin{equation}\label{eqc391}
|m_R \RR\mu - m_Q \RR\mu|\leq M\,\bigr(\Theta(R) + p(Q)\bigl),
\end{equation}
and
\begin{equation}\label{eqc392}
\bigl|\RR(\chi_{R\setminus Q} \mu)(x)\bigr| \leq 2 M \,\bigl (\Theta(R)+ p(Q)\bigr) \qquad \mbox{for all $x\in Q$.}
\end{equation}

\end{itemize}
\end{lemma}

\begin{proof}
The statement (a) is an immediate consequence of the definition above.

Let us turn our attention to (b). The estimate $\Theta(Q)\leq 2B\,\Theta(R)$
is clear if $Q\in\tree_0(R)$. Otherwise, let $Q_0\in\sss_0(R)$ such that
$Q\subset Q_0$. Reasoning by contradiction let us suppose that $\Theta(Q)> 2B\,\Theta(R)$, and let $Q'\in\DD$ be a cube
such that $Q\subset Q'\subset Q_0$ with $\Theta(Q')\geq 2B\,\Theta(R)$ with maximal side length.
Then, $\Theta(P)\leq 2B\,\Theta(R)$ for any $P\in \tree(R)$ such that $Q'\subset P$. 
\begin{align*}
p(Q') & =
\Theta(Q') + \sum_{P:Q'\subsetneq P\subset R} \frac{\ell(Q)}{\ell(P)}\,\Theta(P) 
+  \sum_{P\supsetneq R} \frac{\ell(Q)}{\ell(P)}\,\Theta(P) \\
& \leq \Theta(Q') + c_2\,B\,\Theta(R) + p(R) \\
& \leq \Theta(Q') + c_2\,\Theta(Q') + c_{db}\,\Theta(R)\leq \bigl (1 + c_2 + \frac{c_{db}}B\bigr)\Theta(Q').
\end{align*}
Since $B>2$, if we choose $c_{db}>2(1+c_1)$, we get
$$p(Q')\leq c_{db}\,\Theta(Q').$$
That is, $Q'$ is $p$-doubling, which implies either that $Q'\in\sss(R)$ or there exists another
cube $Q''\in\sss(R)$ which contains $Q'$. This contradicts the fact that
$Q\in\tree(R)$.

To prove that $p(Q)\leq (2B+c_{db})\Theta(R)$ we write
\begin{align}\label{eqalg17}
p(Q) & =
 \sum_{P:Q\subset P\subset R} \frac{\ell(Q)}{\ell(P)}\,\Theta(P) 
+  \sum_{P\supsetneq R} \frac{\ell(Q)}{\ell(P)}\,\Theta(P) \\
& \leq 2B\Theta(R)  + p(R) \leq 2B\Theta(R) + c_{db}\,\Theta(R).\nonumber
\end{align}

The estimate \rf{eqc391} in (c) is a direct consequence of the construction of $\tree_0(R)$. On the other hand, \rf{eqc392} follows from \rf{eqc391} and the inequality \rf{eqfacc12}:
\begin{align*}
\bigl|\RR(\chi_{R\setminus Q}\mu)(x)\bigr| & \leq 
\bigl|m_Q(\RR\mu) - m_R(\RR\mu)\bigr| + p(Q,R)  + 
p(R)\\
&\leq M\,\bigr(\Theta(R) + p(Q)\bigl) + p(Q)+ c_{db}\,\Theta(R) \leq 2M\,\bigr(\Theta(R) + p(Q)\bigl),
\end{align*}
assuming that $M\geq c_{db}$.
\end{proof}

\begin{lemma}\label{lemt0}
Let $R\in\ttt$. We have
\begin{enumerate}[(a)]
\item If $Q\in HD_0(R)$, then $Q$ is $p$-doubling with constant $c_{db}$. Thus $HD(R)=HD_0(R)$.

\item If $Q\in BR(R)$, then
$$|m_Q \RR\mu - m_R \RR\mu|\geq \frac{M}2\,\Theta(R),$$
assuming $M\geq c\,c_{db}$.

\item If $Q\in LD(R)$, then
$$p(Q) \leq c\,\delta_0^{\frac1{s+1}}B^{\frac s{s+1}}\,\Theta(R).$$
\end{enumerate}
\end{lemma}

We will assume that $\delta_0\ll B^{-1}$, so that 
\begin{equation}\label{eqld41}
p(Q) \leq \delta_0^{\frac1{s+2}}\Theta(R)\leq \frac1B \Theta(R)\quad \mbox{ if $Q\in LD(R)$.}
\end{equation}

\vv
\begin{rem}
We will assume that $M=C\, B$, where $C$ is some absolute constant such that the statement in (b) 
of the preceding lemma holds.
\end{rem}

\begin{proof}[Proof of Lemma \ref{lemt0}]
The proof of (a) is already implicit in the proof of part (b) from Lemma \ref{lemfac1}. 
Indeed, similarly to \rf{eqalg17},
\begin{align*}
p(Q)  &=
 \sum_{P:Q\subset P\subset R} \frac{\ell(Q)}{\ell(P)}\,\Theta(P) 
+  \sum_{P\supsetneq R} \frac{\ell(Q)}{\ell(P)}\,\Theta(P) \\
& \leq c_2 \max_{P:Q\subset P\subset R}\Theta(P)  + p(R) \leq 
c_2\,\Theta(Q) + c_{db}\,\Theta(R)\\&\leq (c_2+ c_{db}\,B^{-1})\Theta(Q) \leq c_{db}\,\Theta(Q),\nonumber
\end{align*}
assuming $c_{db}\geq2c_2$ and $B$ big enough in the last inequality. 

For the proof of (b), we know that there exists $Q_{0}\in BR_{0}(R)$ such that $Q\subset Q_{0}$. The following estimate follow from the definition of $ BR_{0}(R)$:
\begin{align*}
|m_Q \RR\mu - m_R \RR\mu| & \geq  |m_{Q_{0}} \RR\mu - m_R \RR\mu|- |m_Q \RR\mu - m_{Q_{0}} \RR\mu|\\
& \geq M\,\bigr(\Theta(R) + p(Q_{0})\bigl) -|m_Q \RR\mu - m_{Q_{0}} \RR\mu|.
\end{align*}
So it suffices to prove $|m_Q \RR\mu - m_{Q_{0}} \RR\mu|\lesssim p(Q_{0})$, since we will choose $M \gg1$. From \rf{eqfacc00}, \rf{eqfacc12} and \rf{eqcad35}, we get
\begin{align*}
|m_Q \RR\mu - m_{Q_{0}} \RR\mu|& \leq |m_Q \RR(\chi_{Q_{0}\setminus Q}\mu)| + |m_Q \RR(\chi_{Q_{0}^{c}}\mu)  - m_{Q_{0}} \RR(\chi_{Q_{0}^{c}}\mu)|\\
& \lesssim \sum_{\substack{P\in \mathcal D \\ Q\subsetneq P \subset Q_{0}}}\frac{\mu(P)}{\ell(P)^{s}} + p(Q_0)
\\
& \lesssim \sum_{j=1}^{\infty} 2^{-j/2}p(Q_{0}) \lesssim p(Q_{0}),
\end{align*}
as wished.

Let us see (c). Let $Q_0$ be the cube from $LD_0(R)$ that contains $Q$. Recall that $\Theta(Q_0)\leq
\delta_0\,\Theta(R)$.
Suppose that
$$\Theta(Q) \geq \tau\,\Theta(R),$$
for some constant $\tau\gg\delta_0$ to be determined below. Consider the largest cube $Q'\in\DD$
with $Q\subset Q'\subset Q_0$ such that $\Theta(Q')\geq\tau\,\Theta(R)$. Then it is immediate
to check that
\begin{equation}\label{eqpp99}
\Theta(Q')\approx\tau\,\Theta(R).
\end{equation} 
Indeed, it is enough to show that $\Theta(Q')\lesssim\tau\,\Theta(R)$, since the opposite inequality is clear from the choice on $Q'$. To prove this, consider the parent $Q''$ of $Q'$. From the fact that
$\ell(Q'')\approx\ell(Q')$ we get
$$\Theta(Q') = \frac{\mu(Q')}{\ell(Q')^s}\leq \frac{\mu(Q'')}{\ell(Q')^s} \approx \Theta(Q'')<\tau\,\Theta(R),$$
as wished. 

We are going to show that $p(Q')\leq c\,\Theta(Q')$. To this end we write
\begin{equation}\label{eqpp100}
p(Q') \leq p(Q',Q_0) + \frac{\ell(Q')}{\ell(Q_0)}\,p(Q_0).
\end{equation}
Since $\Theta(P)\leq \Theta(Q')$ for all $P$ with
$Q'\subset P\subset Q_0$, it follows that
\begin{equation}\label{eqpp101} 
p(Q',Q_0)\lesssim \Theta(Q').
\end{equation}
On the other hand, using the fact that $\mu(Q')\leq \mu(Q_0)$ we derive 
$$\Theta(Q')\leq\frac{\ell(Q_0)^s}{\ell(Q')^s}\,\Theta(Q_0) \leq \delta_0 \,\frac{\ell(Q_0)^s}{\ell(Q')^s}\,\Theta(R),$$
and so, taking int account that
$\Theta(Q')\approx\tau\,\Theta(R)$ we deduce that
$\dfrac{\ell(Q')^s}{\ell(Q_0)^s} \lesssim \dfrac{\delta_0}\tau.$
Therefore, recalling that $p(Q_0)\lesssim B\,\Theta(R)$, we get
$$\frac{\ell(Q')}{\ell(Q_0)}\,p(Q_0) \lesssim\frac{\delta_0^{1/s}}{\tau^{1/s}}\,B\,\Theta(R).$$
Together with \rf{eqpp100} and \rf{eqpp101}, this gives that
$$p(Q')\lesssim \Theta(Q') + \frac{\delta_0^{1/s}}{\tau^{1/s}}\,B\,\Theta(R).$$
Choosing $\tau = \delta_0^{1/(s+1)}\,B^{s/(s+1)}$, we get
$p(Q')\leq c\,\Theta(Q')$,
as wished. So if the constant $c_{db}$ has been taken big enough, $Q'$ is a $p$-doubling cube and
then, by construction, $Q=Q'$, and so the statement (c) follows from \rf{eqpp99}.
\end{proof}


\section{Controlling the size of the trees}

The objective of this section consists in proving the following.

\begin{lemma} \label{lemdif1}
For every $R\in\ttt$, we have
\begin{equation}
\sum_{Q\in\tree_0(R)}\mu(Q)\leq c(B,\delta_0,M)\,\mu(R).\label{eqdif1}
\end{equation}
As a consequence,
\begin{equation}\label{eqdif2}
\sigma(\tree(R))\leq c'(B,\delta_0,M)\,\Theta(R)^2\,\mu(R).
\end{equation}
\end{lemma}

This result is an easy consequence of the next one.

\begin{lemma}\label{lemtouch}
There are constants $\delta,\eta>0$ depending on $B,M,\delta_0$ such the following holds. Given $R\in\ttt$, for any $R'\in\tree(R)$, 
let $\SSS(R')$ denote the subfamily of the cubes $Q\in\sss_0(R)$ which are contained in $R'$ and satisfy
$\ell(Q)\geq\delta\,\ell(R')$. Then, $\SSS(R')$ is non-empty and
$$\mu\biggl (\,\bigcup_{Q\in \SSS(R')}  Q\biggr)\geq \eta\,\mu(R').$$
\end{lemma}

\begin{proof}[Proof of Lemma \ref{lemdif1} using Lemma \ref{lemtouch}]
The estimate \rf{eqdif1} follows by applying the preceding result iteratively. Indeed,
for each $k\geq0$ denote by $\AZ_k$ the collection of the cubes $R'\in\tree_0(R)$ such that
$$\delta^{k+1}\ell(R)< \ell(R') \leq \delta^{k}\,\ell(R)$$
and have maximal side length. Notice that this is a pairwise disjoint family.
From Lemma \ref{lemtouch} it follows that for, any $R'\in\AZ_k$,
$$\mu\biggl (R'\cap \bigcup_{Q\in\AZ_{k+2}:\,Q\subset R'}Q\biggr) \leq (1-\eta)\,\mu(R').$$
Summing over all $R'\in\AZ_k$ we obtain
$$\mu\biggl (\bigcup_{Q\in\AZ_{k+2}}Q\biggr) \leq (1-\eta)\,\mu\biggl (\bigcup_{Q\in\AZ_{k}}Q\biggr) .$$
By iterating this estimate we get
$$\mu\biggl (\bigcup_{Q\in\AZ_{k}}Q\biggr)\leq (1-\eta)^{(k-1)/2}\,\mu(R),$$
and thus
$$\sum_{Q\in\tree_0(R)}\mu(Q)\leq c(\delta)\,\sum_{k\geq0} \sum_{Q\in\AZ_k}\mu(Q)  
\leq c(\delta)\,\sum_{k\geq0}(1-\eta)^{(k-1)/2}\mu(R) \leq c(\delta,\eta)\,\mu(R),$$
which yields \rf{eqdif1}.

The inequality \rf{eqdif2} follows from \rf{eqsigma1}, Lemma \ref{lemfac1} (b) and the fact that
$$\sigma(\tree_0(R))\leq B^2\,\Theta(R)^2\sum_{Q\in\tree_0(R)}\mu(Q)\leq B^2c(B,M,m_0)\,\Theta(R)^2\,\mu(R).$$
\end{proof}

\vv

To prove Lemma \ref{lemtouch} we need to introduce some notation. In the situation of the lemma, given $Q\in
\tree_0(R)$ and an integer $N>1$, we denote 
$$\FF_N(Q)=\tree_0(R)\cap\DD_N(Q)$$ and we consider
the set
$$F_N(Q)= \bigcup_{P\in \FF_N(Q)}P.$$

Also, we consider the following families of cubes, for $N_0\geq N>1$,
$$I_{N_0}(Q) = \biggl\{P\in\tree_0(R):\, P\subset Q, \,\mu(P\cap F_{N_0}(Q))\leq \frac14 \mu(P)\biggr\}$$
and 
$$\wt \FF_{N,N_0}(Q) = \Bigl\{P\in\FF_N(Q):\nexists P'\in I_{N_0}(Q) \text{ such that }P'\supset P\Bigr\}.$$
Then we denote
$$\wt F_{N,N_0}(Q)= \bigcup_{P\in\wt\FF_{N,N_0}(Q)} P = F_N(Q) \setminus \bigcup_{P\in I_{N_0}(Q)}P .$$

\vv
\begin{lemma}\label{lemfnt}
Let $N_0\geq N>1$ be integers. Let $Q\in\tree_0(R)$ be such that $\mu(F_{N_0}(Q))\geq\frac12\,\mu(Q)$.
Then, 
$$\mu(\wt F_{N,N_0}(Q))\geq \mu(\wt F_{N_0,N_0}(Q))\geq\frac14 \,\mu(Q).$$
\end{lemma}

\begin{proof}
The first inequality is trivial. To prove the second one, denote by
$\wt I_{N_0}(Q)$ the subfamily of maximal cubes from $I_{N_0}(Q)$. 
Notice that
$$F_{N_0}(Q)\setminus\wt F_{N_0,N_0}(Q)\subset \bigcup_{P\in \wt I_{N_0}(Q)} P.$$
Since the cubes from $\wt I_{N_0}(Q)$ are disjoint,
we have
$$\mu\bigl (F_{N_0}(Q)\setminus\wt F_{N_0,N_0}(Q)\bigr) \leq \sum_{P\in \wt I_{N_0}(Q)} \mu(F_{N_0}(Q)\cap P)\
\leq \frac14 \sum_{P\in \wt I_{N_0}(Q)} \mu(P)
\leq \frac14\,\mu(Q).$$
Therefore, 
$$\mu(\wt F_{N_0,N_0}(Q)) = \mu(F_{N_0}(Q)) -  \mu\bigl (F_{N_0}(Q)\setminus\wt F_{N_0,N_0}(Q)\bigr) \geq
\frac12 \,\mu(Q) - \frac14\,\mu(Q) = \frac14\,\mu(Q).$$
\end{proof}

\vv
\begin{lemma}\label{lembola}
Let $N\geq 1$. Let $R'\in\tree_0(R)$ be such that $F_N(R')\neq\varnothing$. Then there exists an open ball $B_{R'}\subset\frac{10}9R'$ which satisfies
$$B_{R'}\cap F_N(R')=\varnothing\quad\mbox{ and }\quad \partial B_{R'}\cap F_N(R')\neq\varnothing.$$
\end{lemma}

\begin{proof}
Let $Q_i$, $i=1,\ldots,m$ be the sons of $R'$. By the separation condition \rf{eqsepara}, we know
that there exists some constant $\tau>0$ such that
$$\dist\biggl (Q_1,\bigcup_{i=2}^m Q_i\biggr) \geq \tau\,\ell(Q_1).$$
Thus 
$(1+\tau)Q_1 \setminus Q_1$ does not intersect any cube from $\DD(R')$ different from $R'$. So 
we may take an open ball $B_0$ contained in $R'\cap \bigl ((1+\tau)Q_1 \setminus Q_1\bigr)$
with $r(B_0) \approx \ell(Q_1)\approx \ell(R')$, which clearly 
satisfies $B_0\cap F_N(R')=\varnothing$.

Let $x_1$ be a point from $F_N(R')$ which is closest to the center $x_0$ of $B_0$. We move $B_0$ keeping its center
in the segment $[x_0,x_1]$ till its boundary touches $F_N(R')$, and then we call it $B_{R'}$. Assuming $r\leq \ell(R')/20$, it follows that $B_{R'}\subset 
\frac{10}9R'$.
\end{proof}

\vv

Given a hyperplane $H\subset \R^d$, $x\in H$, $r>0$ and $0<\alpha\leq1/2$, we consider the cone
$$X(x,H,\alpha) = \bigl\{y\in \R^d:\, \dist(y,H)\leq \alpha\,|x-y|\bigr\}.$$
For $0\leq r_1< r_2$, we also set
$$X(x,H,\alpha,r_1,r_2) = \bigl\{y\in \R^d:\, \dist(y,H)\leq \alpha\,|x-y|,\,r_1\leq|x-y|\leq r_2\bigr\}.$$

\begin{lemma}\label{lemcon}
Let $R'\in\tree_0(R)$, $P\in\wt F_N(R')$, and $x\in P$. Given $0<\alpha\leq1/2$, there exists
$m_0=m_0(\alpha)\geq 1$ such that for all $r\geq \ell(P)$ and $m\geq m_0$,
$$\mu(X(x,H,\alpha,\tfrac1{10}\ell(P),r) \cap F_{mN}(R'))\leq c_3\,B\,\alpha^{s+1-d}\,r^s.$$
\end{lemma}

\begin{proof}We assume for simplicity that $x=0$. It is clear that we may also assume that $r\leq\diam R'$.
For $k\geq0$, let $r_k = (1-\alpha)^k\,r$ and consider the closed annulus
$$A_k = A(0,r_{k+1},r_k),$$
so that
$$X(0,H,\alpha,r_{k+1},r_k) = A_k \cap X(0,H,\alpha,0,r).$$

Let $n_0$ be the largest integer such that $r_{n_0}\leq \tfrac1{10}\ell(P)$.
We set
$$\mu\bigl(X(0,H,\alpha,\tfrac1{10}\ell(P),r) \cap F_{mN}(R')\bigr)\leq   \sum_{k=0}^{n_0-1} 
\mu\bigl(X(0,H,\alpha,r_{k+1},r_k)\cap F_{mN}(R')\bigr).$$
We are going to estimate each summand on the right side. To this end, consider the 
$(d-1)$-dimensional annulus $A_k\cap H$. This can be covered by
a family of closed balls $B_1^k,\ldots,B_{n_k}^k$ of radius $\alpha\,r_k$ centered in $A_k\cap H$ with
$$n_k\approx \frac1{\alpha^{d-2}}.$$
 To prove this notice that
the $(d-1)$-dimensional volume of $A_k\cap H$ is comparable to
$$\bigl(r_k-r_{k+1}\bigr)\,r_k^{d-2} = \alpha\,r_k^{d-1},$$
while the one of $B_i^k\cap H$ is $c\,\bigl(\alpha\,r_k\bigr)^{d-1}$ for each $i$. So
$$n_k\approx \frac{\alpha\,r_k^{d-1}}{\bigl(\alpha\,r_k\bigr)^{d-1}} = \frac1{\alpha^{d-2}}.$$

We claim now that 
\begin{equation}\label{eqbik10}
X(0,H,\alpha,r_{k+1},r_k) \subset \bigcup_{i=1}^{n_k} 3B_i^k.
\end{equation}
Indeed, given $y\in X(0,H,\alpha,r_{k+1},r_k)$, let $y'$ be its orthogonal projection on $H$,
so that $|y-y'|=\dist(y,H)$.
Take now 
$$y'' = \frac{|y|}{|y'|}\,y'.$$
Observe that $y''\in H$ and $|y''|=|y|$, and so
$y''\in H\cap X(0,H,\alpha,r_{k+1},r_k)$. Also, we have
$$|y'-y''| =  |y'|\,\biggl(1-\frac{|y|}{|y'|}\biggr) = \bigl| |y'|- |y|\bigr| \leq
|y'-y|.$$
Therefore,
$$|y-y''|\leq |y-y'| + |y'-y''|\leq 2|y-y'|=2\dist(y,H)\leq2\alpha r_k.$$
As a consequence, if $y''\in B_i^k$, then $\dist(y,B_i^k)\leq 2\,r(B_i^k)$ and so
$y\in 3B_i^k$, which proves our claim.

Next we are going to check that
$$
\mu\bigl(3B_i^k\cap F_{mN}(R')\bigr)\leq c\,B\,\Theta(R)\,r(B_i^k)^s = c\,B\,\Theta(R)\,\alpha^s\,
(1-\alpha)^{sk}\,r^s.
$$
for all $0\leq k\leq n_0$ and
$1\leq i \leq n_k$.
To this end notice that for these indices $k,i$,
\begin{equation}\label{eqmes19}
r(B_i^k) = \alpha\,r_k\geq \alpha\,r_{n_0} \geq \alpha\,(1-\alpha)\,r_{n_0-1}\geq
\frac\alpha{20}\,\ell(P)\geq \frac\alpha{20}\,8^{-N}\,\ell(R'),
\end{equation}
recalling that $P\in\FF_N(R')$ in the last inequality.
On the other hand, if $\mu\bigl(3B_i^k\cap F_{mN}(R')\bigr)\neq0$, then there exists some cube $Q\in
F_{mN}(R')$ which intersects $3B_i^k$ and we have
\begin{equation}\label{eqmes20}
\ell(Q)\leq 2^{-mN}\ell(R').
\end{equation}
If $m_0$ is big enough and $m\geq m_0$, from \rf{eqmes19} and \rf{eqmes20} we infer that
$\ell(Q)\leq r(B_i^k)$. Thus there exists some ancestor $Q'$ of $Q$ such that $3B_i^k\subset (1+c_{sep})Q'$ which
satisfies $\ell(Q')\approx r(B_i^k)$. In particular, either $Q'\in\tree_0(R)$ or $Q'$ contains $R$ and
$\ell(Q')\approx\ell(R)$. In any case (using that $R$ is p-doubling in the second one), we deduce
that $\mu(Q')\leq c\,B\,\Theta(R)\,\ell(Q')^s$. So, by the separation property \rf{eqsepara}
$$\mu\bigl(3B_i^k\cap F_{mN}(R')\bigr)\leq \mu((1+c_{sep})Q')= \mu(Q') \lesssim B\,\Theta(R)\,\ell(Q')^s\approx B\,\Theta(R)\,r(B_i^k)^s,$$
as wished.

From the latter estimate we infer that $\mu\bigl(3B_i^k\cap F_{mN}(R')\bigr)\lesssim\,B\,\Theta(R)\,\alpha^s\,
(1-\alpha)^{sk}\,r^s$ for all $i,k$.
Together with \rf{eqbik10} and the fact that $n_k\approx\alpha^{2-d}$, this implies that
$$\mu\bigl(X(0,H,\alpha,r_{k+1},r_k)\cap F_{mN}(R')\bigr)\lesssim B\,\Theta(R)\alpha^{2+s-d}\,(1-\alpha)^{sk}\,r^s.$$
As a consequence,
\begin{align*}
\mu\bigl(X(0,H,\alpha,\tfrac1{10}\ell(P),r) \cap F_{mN}(R')\bigr)&\lesssim  B\,\Theta(R)\,r^s\,\alpha^{2+s-d}\sum_{k=0}^{n_0-1}(1-\alpha)^{sk}\\
& \leq \frac{B\,\Theta(R)\,r^s\,\alpha^{2+s-d}}{1-(1-\alpha)^s}
\approx B\,\Theta(R)\alpha^{1+s-d}\,r^s.
\end{align*}
\end{proof}

\vv
\begin{lemma}\label{lemtouchingpoint}
Suppose that $N$ and $m_0$ are big enough, depending on $\delta_0$ and $B$, and let $N_0$ be such that $N_0\geq m_0\,N$.
Given $R'\in\tree_0(R)$, let $B_{R'}$ be an open ball centered at some point from $R'$ with 
$r(B_{R'})\geq c^{-1}\ell(R')$
such that
$$B_{R'}\cap F_N(R')=\varnothing \quad\mbox{and}\quad \partial B_{R'}\cap \wt F_{N,N_0}(R')\neq\varnothing.$$
Let $P\in \wt\FF_{N,N_0}(R')$ be such that $P\cap \partial B_{R'}\neq \varnothing$.
We have
\begin{equation}\label{eqtou3}
|\RR(\chi_{F_{N_0}(R')\setminus P}\mu)(x)|\geq C_1(B, \delta_{0})\, \Theta(R)\,N\quad\mbox{for all $x\in P$.}
\end{equation}
\end{lemma}
 
\begin{proof}
Using Lemma \ref{lemcomp}, it is easy to check that it is enough to prove the estimate \rf{eqtou3}
for $x\in P\cap \partial B_{R'}$.

Let $\alpha$ be a parameter to be chosen later in the proof. Let $H$ be the
hyperplane which is tangent to $\partial B_{R'}$ at $x$. For simplicity we assume that
$H=\{x\in\R^d:x_d=0\}$ and we suppose that $B_{R'}$ is contained in the half plane $\{y\in\R^d:y_d\leq0\}$.
Consider the cone
$X_\alpha = X(x,H,\alpha)$ and let $V_\alpha = \R^d\setminus X_\alpha$. Also set
$$V_\alpha^+ = \{y\in\R^d:y_d>\alpha\,|x-y|\},\qquad V_\alpha^- =  \{y\in\R^d:y_d<-\alpha\,|x-y|\}
,$$
so that $V_\alpha = V_\alpha^+\cup V_\alpha^-$.
Let $B_0$ be an open ball centered at $x$ such that $B_0\cap V_\alpha^-\subset B_{R'}$ with radius comparable to 
$r(B_{R'})$ with some constant depending on $\alpha$. See Figure~\ref{f.uno}.

\begin{figure}
\begin{tikzpicture}[line cap=round,line join=round,>=triangle 45,x=1.0cm,y=1.0cm]
\clip(-6.72,-3.94) rectangle (9.12,3.48);
\draw(0,0) circle (1.42cm);
\draw (-2,-1)-- (0,0);
\draw (0,0)-- (2,1);
\draw (-2,1)-- (0,0);
\draw (0,0)-- (2,-1);
\draw(0,-1.59) circle (1.59cm);
\draw (-0.18,-0.1) node[anchor=north west] {$x$};
\draw (-2.52,0)-- (2.58,0);
\draw (2.12,0.5) node[anchor=north west] {$\mathcal H$};
\draw (1.5,0.5) node[anchor=north west] {$\alpha$};
\draw (-2.42,0.82) node[anchor=north west] {$X_{\alpha}$};
\draw (0,-0.44) node[anchor=north west] {$ V_\alpha^-$};
\draw (0,1.18) node[anchor=north west] {$ V_\alpha^+$};
\draw (-1.16,2.0) node[anchor=north west] {$B_0$};
\draw (-1.56,-2.8) node[anchor=north west] {$B_{R'}$};
\begin{scriptsize}
\draw [fill=uuuuuu] (0,0) circle (1.5pt);
\draw [fill=uuuuuu] (-1.27,-0.64) circle (1.5pt);
\draw [fill=uuuuuu] (1.27,-0.64) circle (1.5pt);
\end{scriptsize}
\end{tikzpicture}
\caption{The hyperplane $H$, the balls $B_{R'}$, $B_0$, the cone $X_\alpha$, 
and the regions $V_\alpha^+$ and $V_\alpha^-$.}
\label{f.uno} 
\end{figure}
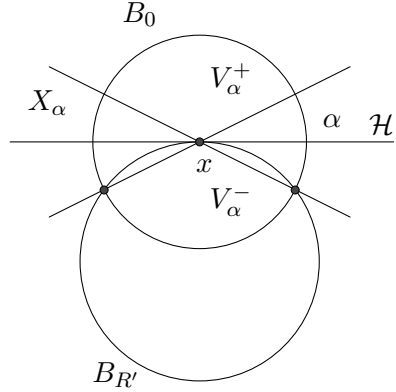

Write
\begin{equation*}
|\RR(\chi_{F_{N_0}(R')\setminus P}\mu)(x)|\geq |\RR(\chi_{(B_0\cap F_{N_0}(R')) \setminus P}\mu)(x)|-|\RR(\chi_{F_{N_0}(R')\setminus B_0}\mu)(x)|
\end{equation*}
We can bound the second term on the right hand side using the size condition on the fractional Riesz transform: 
\begin{equation}\label{e.outside}
|\RR(\chi_{F_{N_0}(R')\setminus B_0}\mu)(x)|\leq \frac{\mu(R')}{r(B_0)^{s}} \leq
 c(\alpha) \frac{\mu(R')}{\ell(R')^{s}} \leq c(\alpha) B\,\Theta(R).
\end{equation}

We focus on the term $|\RR(\chi_{(B_0\cap F_{N_0}(R'))\setminus P}\mu)(x)|$. We will use that $(B_0\cap F_{N_0}(R'))\cap  V_\alpha^+ =\varnothing$, by the construction of $B_0$, and that $x_{d}-y_{d}>0$
in $V_\alpha^-$. Denoting by $\RR^d$ the $d$-th component of $\RR$, we have
\begin{align*}
\RR^d(\chi_{B_0\setminus P}\mu)(x) & = \RR^{d}(\chi_{(B_0\cap F_{N_0}(R')\setminus P)\cap V_\alpha^+}\mu)(x) + \RR^{d}(\chi_{(B_0\cap F_{N_0}(R')\setminus P)\cap V_{\alpha}^-}\mu)(x) \\&\quad+\RR^{d}(\chi_{(B_0\cap F_{N_0}(R')\setminus P)\cap X_{\alpha}}\mu)(x) \\
&  \geq \RR^{d}(\chi_{(B_0\cap F_{N_0}(R')\setminus P)\cap V_{\alpha}^-}\mu)(x) -
|\RR^{d}(\chi_{(B_0\cap F_{N_0}(R')\setminus P)\cap X_{\alpha}}\mu)(x)|
\\
&\geq \int_{B_0\cap F_{N_0}(R')\cap V_\alpha^-\setminus P}\!\frac{x_{d}-y_{d}}{|x-y|^{s+1}}d\mu(y)- \!\!\int_{B_0\cap F_{N_0}(R')\cap X_{\alpha}\setminus P}\!
\frac{|x_{d}-y_{d}|}{|x-y|^{s+1}}d\mu(y)\\
&\geq \int_{B_0\cap F_{N_0}(R')\cap V_{\alpha}^-\setminus P}\frac{\alpha}{|x-y|^s}d\mu(y)- \int_{B_0\cap F_{N_0}(R')\cap X_{\alpha}\setminus P}\frac{\alpha}{|x-y|^s}d\mu(y)\\
&=\int_{B_0\cap F_{N_0}(R')\setminus P}\frac{\alpha}{|x-y|^s}d\mu(y)- 2\int_{B_0\cap F_{N_0}(R')\cap X_{\alpha}
\setminus P}\frac{\alpha}{|x-y|^s}d\mu(y)\\
& = (I)-(II).
\end{align*}

We will study each of $(I)$ and $(II)$ in turn. First, we will bound $(II)$ from above. To this end let $P_0\in\DD$ be the smallest cube such that $1.1P_0$ contains $B_0$. 
It is clear that 
$\ell(P_0)\approx r(B_0)$. Given any $Q\in\DD$, $Q^1$ stands for the parent of $Q$.
Using Lemma \ref{lemcon} then we get
\begin{align*}
(II) & \leq 2\sum_{Q:P\subset Q\subsetneq P_0}\int_{( Q^1\setminus Q)\cap F_{N_0}(R')\cap X_\alpha}
\frac{\alpha}{|x-y|^s}d\mu(y)\\
& \leq c\sum_{Q:P\subset Q\subsetneq P_0}
\frac{\alpha\,\mu\bigl(( Q^1\setminus Q)\cap F_{N_0}(R')\cap X_\alpha\bigr)}{\ell(Q)^s}\nonumber
\\ 
& \leq c
\sum_{Q:P\subset Q\subsetneq P_0} \frac{B\,\Theta(R)\,\alpha^{2+s-d}\ell( Q^1)^s}{\ell(Q)^s}
\nonumber\\
&\approx B\,\Theta(R)\,\alpha^{2+s-d}\,n_0,\nonumber
\end{align*}
where $n_0= \#\{Q\in\DD:P\subset Q\subsetneq P_0\}$.
Notice that in the third inequality we assumed $m_0(\alpha)$ big enough, so that the assumptions of Lemma \ref{lemcon} are satisfied.

We now look for a lower bound for $(I)$. First we fix a integer $k_{0}$ to be chosen later in the proof, depending on $B$ and $\delta_0$. For $j\geq1$, we denote by $P^{j}$, the $j$-th ancestor of $P$ in the Cantor construction.
Then we have
\begin{align}\label{eqI}
(I) & \geq  \sum_{k=1}^{n_1}\int_{(P^{k_{0}(k+1)}\setminus P^{k_{0}k})
\cap F_{N_0}(R')}
\frac{\alpha}{|x-y|^s}d\mu(y)\\
& \geq c(k_0)^{-1} \sum_{k=1}^{n_1}\frac{\alpha\,\mu\bigl(P^{k_{0}(k+1)}
\cap F_{N_0}(R')\setminus P^{k_{0}k}\bigr)}{\ell(P^{k_{0}(k+1)})^s},\nonumber
\end{align}
where $n_1$ is the largest integer such that $P^{k_{0}(n_1+1)}\subset B_0$. 

To estimate $\mu\bigl(P^{k_{0}(k+1)}
\cap F_{N_0}(R')\setminus P^{k_{0}k}\bigr)$ from below notice that $P^{k_{0}(k+1)}\not\in
I_{N_0}(R')$, because otherwise $P\not\in \wt \FF_{N,N_0}(R')$. This tells us that
$$\mu\bigl(P^{k_{0}(k+1)}\cap F_{N_0}(R')\bigr)>\frac14\,\mu\bigl(P^{k_{0}(k+1)}\bigr)
\geq \frac14\,\delta_0\,\Theta(R)\,\ell(P^{k_{0}(k+1)}\bigr)^s.$$
Taking into account that 
$$\mu\bigl(P^{k_{0}k}\bigr)\leq B\Theta(R)\,\ell(P^{k_{0}k}\bigr)^s
\leq B\Theta(R)\,2^{-k_0s}\,\ell(P^{k_{0}(k+1)}\bigr)^s,$$
we deduce that for $k_0$ big enough, depending on $B$ and $\delta_0$,
$$\mu\bigl(P^{k_{0}(k+1)}\cap F_{N_0}(R')\bigr)\geq 2\,\mu\bigl(P^{k_{0}k}\bigr),$$
and thus
$$\mu\bigl(P^{k_{0}(k+1)}
\cap F_{N_0}(R')\setminus P^{k_{0}k}\bigr)\geq \frac12\,\mu\bigl(P^{k_{0}(k+1)}\cap F_{N_0}(R')\bigr)\geq \frac18\,\delta_0\,\Theta(R)\,\ell(P^{k_{0}(k+1)}\bigr)^s.$$
Plugging this estimate into \rf{eqI} we obtain
$$(I)\geq c(k_0,\delta_0)^{-1}\,\alpha\,\Theta(R)\,n_1.$$
Taking into account that $n_1\approx n_0/k_0$, we get
$$(I)\geq c(B,\delta_0)^{-1}\,\alpha\,\Theta(R)\,n_0.$$
Together with the upper bound we found for (II), this gives
$$\RR^d(\chi_{B_0\setminus P}\mu)(x)\geq \bigl[
c(B,\delta_0)^{-1}\,\alpha\,- c\,B\,\alpha^{2+s-d}\bigr]\Theta(R)\,n_0.$$
For $\alpha$ small enough depending on $B$ and $\delta_0$, the first factor on the right side is 
positive, and then we get
$\RR^d(\chi_{B_0\setminus P}\mu)(x)\geq c(\delta_0,B)^{-1}\Theta(R)n_0$. 
This implies \rf{eqtou3}, taking into account that
$n_0$ is comparable to $N$ (with some constant that may depend on $\alpha$, and so on $\delta_0$ and $B$).
\end{proof}

\vv

\begin{lemma}\label{lemcub2}
Let $M'>M$ be some large constant.
 Suppose that $N_0$ is a sufficiently large integer, depending on $B$ and $M'$. 
 For all $R'\in\tree_0(R)$, one of the two following conditions holds:
 \begin{itemize}
 \item[(a)] The union of the cubes from $\sss_0(R)$ which belong to $\bigcup_{j=1}^{N_0+1} \DD_j(R')$
  has $\mu$-measure larger that $\mu(R')/4$.
 
 \item[(b)] There exists a family of pairwise disjoint cubes $\TT(R')\subset \bigcup_{j=1}^{N_0+1} \wt \FF_{j,N_0}(R')$ with
 \begin{equation}\label{eqrest53}
\mu\biggl (\,\bigcup_{P\in\TT(R')} P\biggr)\geq \frac3{20}\,\mu(R')
\end{equation}
 such that, for each $P\in\TT(R')$,
  \begin{equation*}
 |\RR(\chi_{F_{N_0}(R')\setminus P}\mu)(x)|\geq M'\, \Theta(R)\quad\mbox{ for all $x\in P$.}
 \end{equation*}
 \end{itemize}
\end{lemma}

\begin{proof}
Suppose that (a) does not hold. This means that
$$\mu\bigl (F_{N_0}(R')\bigr) \geq \frac34\,\mu(R').$$
To show that (b) holds we choose $N_0=2n\, N$, where $n$ and $N$ are some big integers to be fixed below. In particular, we will require $n\geq m_0$, where $m_0=m_0(B,\delta_0)$ is as in 
Lemma \ref{lemtouchingpoint}. Also we will assume that $N$ is big enough so that the same lemma holds and, moreover,
$$C_1(B, \delta_{0})\,N \geq 3M',$$
where $C_1(B,\delta_0)$ is the constant in \rf{eqtou3}.

We define $\TT(R')$ as the subfamily of cubes $P$ from $\bigcup_{j=1}^{N_0+1} \wt \FF_{j,N_0}(R')$ such that
$$|\RR(\chi_{F_{N_0}(R')\setminus P}\mu)(x)|\geq M'\, \Theta(R)\quad\mbox{ for all $x\in P$}$$
and moreover have maximal side length.
To prove \rf{eqrest53}
we define an auxiliary family $\TT_{aux}(R')$ of pairwise disjoint cubes from $\bigcup_{j=1}^{N_0} \wt \FF_{j,N_0}(R')$
by a repeated application of a touching point argument, as follows. We will inductively construct
families $\TT_{aux}^j(R')$, $j=1,\ldots,n$. In the case $j=1$, by Lemma \ref{lembola},
there exists a ball $B_{R'}$ contained in
$\frac{10}9R'$, with radius $\geq c^{-1}\,\ell(R')$ which satisfies
$$B_{R'}\cap F_N(R')=\varnothing\quad\mbox{ and }\quad \partial B_{R'}\cap F_N(R')\neq\varnothing.$$
Let $P\in \FF_{N}(R')$  be such that $\partial B_{R'}\cap P\neq\varnothing$. 
We set
$$\TT_{aux}^1(R') = \{P\}.$$

Assume now that $\TT_{aux}^1(R'),\ldots,\TT_{aux}^{j}(R')$ have already been defined and
let us construct $\TT_{aux}^{j+1}(R')$. 
For each $Q\in F_{jN}(R')$ which intersects $F_{N_0}(R')$ and is not contained in any cube from 
$\TT_{aux}^1(R')\cup\ldots\cup\TT_{aux}^{j}(R')$, consider a ball $B_Q$ 
 contained in
$\frac{10}9Q$, with radius $\geq c^{-1}\,\ell(Q)$ which satisfies
$$B_Q\cap F_N(Q)=\varnothing\quad\mbox{ and }\quad \partial B_Q\cap F_N(Q)\neq\varnothing.$$
Let $P_Q\in \FF_N(Q)\subset \FF_{(j+1)N}(R')$ be such that $\partial B_Q\cap P_Q\neq\varnothing$. 
We set
$$\TT_{aux}^{j+1}(R') = \{P_Q\}_Q,$$
where $Q$ ranges over all the cubes described above.
Notice that the cubes from $\TT_{aux}(R')= \TT_{aux}^1(R')\cup\ldots\cup\TT_{aux}^{n}(R')$ are pairwise disjoint,
by construction.

We claim that if $m_0$ is sufficiently big, then
\begin{equation}\label{eqrest55}
\mu\biggl (\bigcup_{P\in\TT_{aux}(R')} P\biggr)\geq \frac9{10}\,\mu(F_{N_0}(R')).
\end{equation}
To check this, observe first that there is some constant $\delta>0$ depending on $N$, $B$ and $\delta_{0}$ such that
$$\mu(P_Q)\geq \delta\mu(Q),$$
for $Q$ and $P_Q$ as in the previous paragraph. Indeed, $\ell(P_Q)\geq 8^{-N}\ell(Q)$, and thus
$$\mu(P_Q)\geq \delta_{0}\,\ell(P_Q)^s \Theta(R)\geq 8^{-Ns}\delta_{0}\,\ell(Q)^s \,\Theta(R)\geq 
8^{-Ns}\delta_{0}B^{-1}\,\mu(Q) =:\delta\,\mu(Q).$$
Then, from the above construction, it turns out that
$$\mu\biggl (\,\bigcup_{P\in\TT_{aux}^j(R')}P\biggr) \geq \delta\,\biggl[\mu(F_{N_0}(R')) - 
\mu\biggl (\,\bigcup_{k=1}^{j-1}\bigcup_{P\in\TT_{aux}^k(R')}P\biggr)\biggr].$$
Summing this estimate over $1\leq j\leq n$, we obtain
\begin{align*}
\mu\biggl (\bigcup_{P\in\TT_{aux}(R')}P\biggr) & \geq \delta n\,\mu(F_{N_0}(R')) - 
\delta \sum_{j=1}^n\mu\biggl (\bigcup_{k=1}^{j-1}\bigcup_{P\in\TT_{aux}^k(R')}P\biggr)\\
& \geq \delta n\,\mu(F_{N_0}(R')) - 
\delta n \,\mu\biggl (\bigcup_{P\in\TT_{aux}(R')}P\biggr).
\end{align*}
Thus,
$$\mu\biggl (\bigcup_{P\in\TT_{aux}(R')}P\biggr)  \geq \frac{\delta n}{1+\delta n}\,\mu(F_{N_0}(R')).$$
So, if $n$ is big enough, the claim \rf{eqrest55} follows.

To prove \rf{eqrest53}, denote now by $\TT_0(R')$ the subfamily of the cubes $Q\in\TT_{aux}(R')$ which contain some cube
$P\in\wt \FF_{N_0,N_0}(R')$. 
Notice that such cubes $Q$ belong to $\wt \FF_{jN,N_0}(R')$ for some $j\in[1,n]$.
So if $Q\in \TT_{aux}(R')\setminus \TT_0(R')$, 
then 
$$\supp\mu\cap Q\subset F_{N_0}(R')\setminus \wt F_{N_0,N_0}(R').$$
Together with Lemma \ref{lemfnt} this yields
\begin{align*}
\mu\biggl (\,\bigcup_{Q\in \TT_{aux}(R')\setminus \TT_0(R')}Q\biggr) & \leq 
\mu\bigl (F_{N_0}(R')\setminus \wt F_{N_0,N_0}(R')\bigr) \\
&\leq \mu(F_{N_0}(R')) - \frac14 \mu(R')\leq \frac34\,\mu(F_{N_0}(R')).
\end{align*}
Thus,
\begin{align*}
\mu\biggl (\,\bigcup_{Q\in \TT_0(R')}Q\biggr) & = \mu\biggl (\bigcup_{Q\in\TT_{aux}(R')} Q\biggr)
- \mu\biggl (\,\bigcup_{Q\in \TT_{aux}(R')\setminus \TT_0(R')}Q\biggr)\\
& \geq 
\frac9{10}\,\mu(F_{N_0}(R'))- \frac34\,\mu(F_{N_0}(R')) = \frac3{20}\,\mu(F_{N_0}(R')).
\end{align*}

On the other hand,  from the construction above and Lemma \ref{lemtouchingpoint}, it turns out that
for each $P\in\TT_0(R')\cap \FF_{jN}(R')$, with $1\leq j\leq n$, there exists some cube $P'\in\tree_0(R')\cap\FF_{(j-1)N}(R')$ with $P\subset P'\subset R'$
such that
$$|\RR(\chi_{P'\cap F_{N_0}(R')\setminus P}\mu)(x)|\geq 3M'\, \Theta(R)\quad\mbox{ for all $x\in P$.}$$

Indeed, notice that $P\in\wt\FF_{N,N_0-(j-1)N}(P')$ because $P\in\TT_0(R')$ and we have
$$N_0-(j-1)N\geq N_0-nN=nN\geq m_0\,N,$$ so that the assumptions in Lemma \ref{lemtouchingpoint} hold.

We distinguish now two cases. First, if 
$$|\RR(\chi_{F_{N_0}(R')\setminus P'}\mu)(x)|\leq 2M'\, \Theta(R)\quad\mbox{ for all $x\in P'$,}$$
then
\begin{align*}
|\RR(\chi_{F_{N_0}(R')\setminus P}\mu)(x)| & \geq 
|\RR(\chi_{P'\cap F_{N_0}(R')\setminus P}\mu)(x)| - |\RR(\chi_{F_{N_0}(R')\setminus P'}\mu)(x)|\\
& \geq
 M'\, \Theta(R)\quad\mbox{ for all $x\in P$,}
\end{align*}
so that $P$ belongs to $\TT(R')$.
In the second case, if 
$$|\RR(\chi_{F_{N_0}(R')\setminus P'}\mu)(x)|\geq 2M'\, \Theta(R)\quad\mbox{ for some $x\in
P'$,}$$
then we deduce that, for all $y\in P'$,
$$|\RR(\chi_{F_{N_0}(R')\setminus P'}\mu)(y)|\geq 2M'\, \Theta(R) - c\,p(P')
\geq 2M'\, \Theta(R) - C(B)\,\Theta(R) \geq M'\,\Theta(R),$$
assuming that $M'$ is big enough.
So $P'\in\TT(R')$, and thus $P\subset \bigcup_{Q\in\TT(R')} Q$.
Then we infer that
$$\mu\biggl (\,\bigcup_{P\in\TT(R')} P\biggr)\geq \mu\biggl (\,\bigcup_{P\in\TT_0(R')} P\biggr)\geq \frac3{20}\,\mu(R').$$
\end{proof}

\vv
\begin{lemma}\label{lemrieszmax}
We have
$$\|\RR(\chi_{R'\setminus F_{N_0}(R')}\mu)\|_{L^2(\mu\rest F_{N_0}(R'))} \leq C(B,M)\,\Theta(R)\,\mu(R')^{1/2}.$$
\end{lemma}

\begin{proof} We will use the technique of the suppressed kernels of Nazarov, Treil and Volberg. 
Consider the set 
$$H = \bigcup_{Q\in\sss_0(R)} (1+\tfrac14c_{sep})Q.$$
For $x\in\R^d$, denote $\Phi(x) = \dist(x,H^c)$. Observe that $\Phi$ is a $1$-Lipschitz function, and
 if $x\in Q\in\sss_0(R)$, then $\Phi(x)\approx\ell(Q)$. Recall that every $Q'\in\tree_0(R)$
 satisfies $\Theta(Q')\leq B\Theta(R)$. Then it easily follows that
 all the balls $B(x,r)$ with $x\in R$, $\Phi(x)\leq r\leq\ell(R)$, such that
$$\mu(B(x,r))> C_3\,B\,\Theta(R)\,r^s$$
are contained in $H$ if $C_3$ is some sufficiently big constant.

For $x, y\in\R^d$ we consider the kernel $K_\Phi(x,y)$ defined in \rf{eqsuppressed} and, 
given the measure $\sigma =\mu\rest R$,  we consider the operator $\RR_{\Phi,\sigma}$.
By \rf{eqc392} we know that given $x\in R$, for all $Q'\in\tree_0(R)\cup \sss_0(R)$ such that $x\in Q'$,  $\bigl|\RR(\chi_{R\setminus Q'} \mu)(x)\bigr| \leq C(B,M) \,\Theta(R)$.
Using the separation condition \rf{eqsepara}, then we infer that
$$|\RR_{\ve} \sigma(x)| \leq C(B,M)\,\Theta(R) \quad\mbox{for $x\in Q'\in\sss_0(R)$ and for all $\ve>c^{-1}\ell(Q)$}.$$  And also
$$|\RR_{\ve} \sigma(x)| \leq C(B,M)\,\Theta(R) \quad\mbox{for $x\in Q'\in R\setminus\sss_0(R)$ and for all $\ve>0$}.$$

Notice that if $x$ belongs to some stopping cube $Q$ from $\sss_0(R)$, then
$\Phi(x)\approx \ell(Q)$. 
As a consequence, it follows that
$$\sup_{\ve\geq \Phi(x)}|\RR_\ve \sigma(x)| \leq C(B,M)\,\Theta(R) \quad\mbox{for $x\in R$.}$$ 
By Theorem \ref{teontv} we deduce that
$\RR_{\Phi,\sigma}$ is bounded in $L^2(\sigma)$, 
 with its norm not exceeding $C(B,M)\,\Theta(R)$. Therefore,
\begin{equation}\label{eqfaj21}
\|\RR_{\Phi,\sigma}\chi_{R'\setminus F_N(R')} \|_{L^2(\sigma\rest F_N(R'))}
\leq c(B,M)\,\Theta(R)\,\mu(R')^{1/2}.
\end{equation}

To estimate $\|\RR(\chi_{R'\setminus F_{N_0}(R')}\sigma)\|_{L^2(\sigma\rest F_{N_0}(R'))}$, observe that
if $y\in \supp\,\sigma\cap R'\setminus F_{N_0}(R')$ and $x\in \supp\,\sigma\cap F_{N_0}(R')$, then $y$ belongs to some cube $Q \in\sss_0(R)$, 
while $x\not\in Q$. Then it follows that
$$|x-y|\gtrsim \bigl (\Phi(x) + \Phi(y)\bigr).$$
So, 
\begin{align*}
|\RR(\chi_{R'\setminus F_{N_0}(R')}\sigma)(x)|&\leq |\RR_{\Phi,\sigma}(\chi_{R'\setminus F_{N_0}(R')})(x)|
+ c\,\sup_{r\geq \Phi(x)}\frac{\mu(B(x,r)\cap F_{N_0}(R'))}{r^s}\\
&\leq |\RR_{\Phi,\sigma}(\chi_{R'\setminus F_{N_0}(R')})(x)| + C(B)\,\Theta(R).
\end{align*}
Together with \rf{eqfaj21} this implies that
$$\|\RR(\chi_{R'\setminus F_N(R')}\sigma)\|_{L^2(\sigma\rest F_N(R'))} \leq C(B,M)\,\Theta(R)\,\mu(R')^{1/2}.$$
\end{proof}

\vv
\begin{proof}[\bf Proof of Lemma \ref{lemtouch}]
If the condition (a) from Lemma \ref{lemcub2} holds we are done. Otherwise, we consider the 
family $\TT(R')$ in the statement (b) of that lemma and we set
$$G=\bigcup_{P\in\TT(R')} P.$$
By Lemma \ref{lemrieszmax}, there exists some constant $c_5=c_5(B,M)$ such that the set
$$V=\bigl\{x\in F_{N_0}(R'):|\RR(\chi_{R'\setminus F_{N_0}(R')}\mu)(x)|>c_5\Theta(R)\bigr\}$$
satisfies 
$$\mu(V)\leq \frac18\mu(R').$$
Then, 
$$\mu(G\setminus V)\geq \biggl (\frac14 - \frac1{8}\biggr)\mu(R')= \frac1{8}\,\mu(R').$$
Moreover, for $x\in G\setminus V$, if we denote by $P_x$ the cube from $\TT(R')$ that contains $x$,
$$|\RR(\chi_{R'\setminus P_x}\mu)(x)| \geq 
|\RR(\chi_{F_{N_0}(R')\setminus P_x}\mu)(x)| - \RR(\chi_{R'\setminus F_{N_0}(R')}\mu)(x)
\geq (M'-c_3)
\, \Theta(R).$$
From Lemma \ref{lemcomp} we deduce that
\begin{align*}
|m_{P_x} \RR\mu - m_{R'}\RR\mu| & \geq |\RR(\chi_{R'\setminus P_x}\mu)(x)| - c\,p(P_x) - 
c\,p(R') \\
& \geq (M' - c_5- c\,B)\,\Theta(R)\geq 3M\,\Theta(R),
\end{align*}
assuming $M'$ big enough in the last inequality.
However, since $R',P_x\in\tree_0(R)$, we have
$$|m_{P_x} \RR\mu - m_{R'}\RR\mu|\leq |m_{P_x} \RR\mu - m_{R}\RR\mu| + 
|m_{R} \RR\mu - m_{R'}\RR\mu| \leq 2M\Theta(R),$$
which is a contradiction. Thus (a) from Lemma \ref{lemcub2} holds.
\end{proof}


\section{Types of trees}

Recall that given a collection of cubes $\AZ\subset\DD$, $\sigma(\AZ)$ stands for
$$\sigma(\AZ) = \sum_{Q\in\AZ}\Theta(Q)^2\mu(Q).$$
So Theorem \ref{teopri} consists in proving that
$$\|\RR \mu\|_{L^2(\mu)}^2 \approx  \sigma(\DD),$$
under the assumption that
$$\sup_{Q\in \DD} \Theta(Q)\lesssim1.$$

For the proof we need to consider different types of trees. We say that $\TT$ is a simple tree if it is of the form $\TT=\tree(R)$, with $R\in \ttt$, (defined in Section \ref{seccorona}). 
Given a small constant $\delta_W>0$, we say that a simple tree $\TT=\tree(R)$, $R\in\ttt$, is of type
$W$ (wonderful) if 
\begin{equation*}
\mu\biggl (\,\bigcup_{Q\in BR(R)} Q\biggr)\geq \delta_W\,\mu(R).
\end{equation*}

If $\TT$ is not of type $W$, we say that it is of type $NW$.

Let $A>10$ to be fixed below. For the reader's convenience, let us remark that we will choose
$$\delta_0,\delta_{W}\ll 1\ll M\ll A\ll B.$$ 
For instance, we may choose $A=B^{1/100}$.
We say that a simple tree $\TT$ is of type $I\sigma$ (increasing $\sigma$) if it is type $NW$ and
$$\sigma(\sss(R))>A\,\sigma(\{R\}).$$
We say that $\TT$ is of type $D\sigma$ (decreasing $\sigma$) if it is of type $NW$ and
$$\sigma(\sss(R))< A^{-1}\,\sigma(\{R\}).$$
On the other hand, we say that $\TT$ is of type $S\sigma$ (stable $\sigma$) if it is of type $NW$ and
$$A^{-1}\,\sigma(\{R\})\leq\sigma(\sss(R))\leq A\,\sigma(\{R\}).$$


Our analysis is going to need more complex trees, in fact, we will use three families of trees: the family of maximal decreasing trees $MDec$, the family of large wonderful trees denoted by $LW$ and the family of tame increasing trees $TInc$. Let us  describe these three families in turn.

\vv
{\bf(1)} Given a $D\sigma$ simple tree $\TT=\tree(R)$ we construct an $MDec-enlarged$ tree $\widetilde{\TT}$ iteratively as follows. We set 
\begin{equation*}
\widetilde{\TT}_{1}:= \TT \cup \bigcup_{\substack{P \in \sss(\TT)\\ \tree(P) \text{ is } D\sigma }} \tree(P).
\end{equation*}
We also define the stopping collection of the new tree as
\begin{equation*}
\sss(\widetilde{\TT}_{1}):= \{P\in\sss(\TT):\text{$\tree(P)$ is not $D\sigma$}\}
\cup \!\!\bigcup_{\substack{P \in \sss(\TT)\\ \tree(P) \text{ is } D\sigma }}\!\! \sss(P),
\end{equation*}
where $\sss(\TT)\equiv\sss(R)$.
In general, for an integer $k\geq 2$, we define 
$$
\widetilde{\TT}_{k}:=\widetilde{\TT}_{k-1}\cup \bigcup_{\substack{P \in \sss(\widetilde{\TT}_{k-1})\\ \tree(P) \text{ is } D\sigma }} \tree(P),
$$
and 
$$
\sss(\widetilde{\TT}_{k}):=
\{P\in\sss(\widetilde{\TT}_{k-1}):\text{$\tree(P)$ is not $D\sigma$}\}
\cup  \!\!\bigcup_{\substack{P \in \sss(\widetilde{\TT}_{k-1})\\ \tree(P) \text{ is } D\sigma }} \!\!\sss(P).
$$

Our $MDec-enlarged$ tree $\widetilde{\TT}$ will be the maximal  union of  $D\sigma$ trees, that is, 
\begin{equation*}
\widetilde{\TT}:=\bigcup_{k\geq 1} \widetilde{\TT}_{k}
\end{equation*}



\vv

{\bf(2)} Given a simple tree $\TT=\tree(R)$ that is either $I\sigma$ or $S\sigma$, the following algorithm will return the desired  $LW$ or $TInc$ tree. We first define the tree $\wt \TT_{2}$ as
\begin{equation*}
\wt\TT_{2}:= \TT \cup \bigcup_{P\in HD(R)}\tree(P).
\end{equation*}
We define the stopping collection of the new tree $\wt\TT_{2}$ as 
$$\sss(\wt\TT_{2}):=\bigl(\sss(\TT)\setminus HD(R)\bigr) \cup \bigcup_{P\in HD(R)} \sss(P).$$ The collections of HD, LD and BR cubes are, respectively,
 $$HD_{2}:=HD(\wt\TT_{2}):= \bigcup_{P\in HD(R)} HD(P),$$
$$LD_{2}:=\bigcup_{P\in HD(R)} LD(P) \quad \text{ and }\quad BR_{2}:=\bigcup_{P\in HD(R)} BR(P).$$


In general, for an integer $k\geq 2$, if $\wt\TT_{k-1}$ has already been defined
and moreover the following two conditions hold:
$$\sigma(\sss(\wt\TT_{k}))> A \sigma(\sss(\wt\TT_{k-1}))
\quad\text{ and }\quad
\mu(BR_{k})\leq \delta_{W}' \mu(HD_{k-1}),$$
then we define 
$$
\widetilde{\TT}_{k}:=\widetilde{\TT}_{k-1}\cup \bigcup_{P\in HD_{k-1}} \tree(P),
$$
and 
\begin{align*}
\sss(\widetilde{\TT}_{k})& :=
\{P\in\sss(\widetilde{\TT}_{k-1})\setminus HD_{k-1}\}
\cup  \!\!\bigcup_{P\in HD_{k-1}} \!\!\sss(P).
\end{align*}
Moreover, we denote
$$HD_{k}:=\!\bigcup_{P\in HD_{k-1}} \!\!HD(P), \quad  LD_{k}:=\!\bigcup_{P\in HD_{k-1}}\!\!LD(P)\quad \text{ and }\quad BR_{k}:=\!\bigcup_{P\in HD_{k-1}}\!\! BR(P).$$

We stop the algorithm when we reach one of the following conditions for some $k=N_0$: 
\begin{enumerate}
\item[(a)] If \begin{equation} \label{e.stoplw} \mu(BR_{N_{0}})> \delta_{W}' \mu(HD_{N_{0}-1}) \end{equation} for certain $N_{0}\geq 2$, in which case we say that $\wt\TT_{N_{0}}$  belongs to the $LW$ family and we say that $N_{0}$ is its order.

\item[(b)] Or if 
\begin{equation}\label{e.tame}
\sigma(\sss(\wt\TT_{N_{0}}))\leq A \sigma(\sss(\wt\TT_{N_{0}-1}))
\end{equation}
 for certain $N_{0}\geq 2$. The boundedness of densities ensures that such condition must be reached at some $N_{0}$. Then $\wt\TT_{N_{0}}$ belongs to the $TInc$ family and we say that $N_{0}$ is its order.

\end{enumerate}

\vv

The starting cube $R$ in the construction of a tree $\wt\TT$ of type $MDec$, $TInc$, $LW$, or $W$ is
called {\em root of $\wt\TT$}, and we write $R=\roo(\wt\TT)$.

We say that the trees of type $MDec$, $TInc$, $LW$, and $W$ are {\em maximal trees}. 
Note that the trees of type $W$ are simple. Maximal trees of type $MDec$ can be simple or non-simple. On the other hand, trees of type $TInc$ and $LW$ are non-simple.


\subsection{Estimates for the trees of type $W$ and $LW$}

Throughout this section we will use that the dyadic density $\Theta_{d}(Q)$ is constant for all cubes $Q\in HD_{k}$ for a fixed $k\in\mathbb N$ and we will denote such constant by $\Theta(HD_{k})$. In order to prove the desired estimates for trees of type $LW$ and $TInc$ we need the following lemma:

\begin{lemma} \label{l.sigmatinc}
Let $\wt \TT \in TInc\cup LW$ be of order $N_{0}$. Then
$$\sigma(\wt \TT)\leq C(A,B, \delta_{0}, \delta_{W})  \sigma(HD_{N_{0}-1}).$$
\end{lemma}

\begin{proof}
First, we observe that by Lemma \ref{lemdif1}, 
$$\sigma(\wt \TT) \leq C(A,B, \delta_{0}) \sum_{k=0}^{N_{0}-1}\,\sigma(HD_{k}),$$
 where, by convenience, we set $HD_{0}:=R$, where $R$ is the root of $\wt \TT$. 

The lemma follows now trivially by the geometrically increasing nature of $\sigma(HD_{k})$. By iteration, it is enough to prove that for $2\leq k\leq N_{0}-1$
\begin{equation}
\label{e.dechd}
\sigma(HD_{k-1})\leq \frac2A\, \sigma(HD_{k}),
\end{equation}
recalling that $2A^{-1} <1$. Notice also that for $k=1$
\begin{equation} \label{e.uglycontrol}
\sigma(R)\leq 2A\sigma(HD_{1}).
\end{equation}

Since $2\leq k\leq N_{0}-1$, by the definition of $TInc$ and $LW$ tree $\wt \TT$, we have that $\sigma(\sss(\wt \TT_{k}))>A\,\sigma(\sss(\wt \TT_{k-1}))$. Thus
\begin{align}\label{eqfa520}
\sigma(HD_{k})+ \sigma(BR_{k}) +\sigma(LD_{k}) &= \sigma(\sss(\widetilde{\TT}_{k})\setminus
\sss(\widetilde{\TT}_{k-1}))\\
&\geq (A-1)\,\sigma(\sss(\widetilde{\TT}_{k-1}))\geq (A-1)\,\sigma(HD_{k-1}).\nonumber
\end{align}
On the other hand for $k=1$, $\wt \TT_{1}$ is of either type $I\sigma$ or $S\sigma$, then we have
\begin{equation}\label{e.hd1}
\sigma(\sss(\widetilde{\TT}_{1}))\geq A^{-1}\sigma(R).
\end{equation}
We also have that for all $k\geq 1$
$$
\sigma(LD_{k})\leq \delta_{0}^{2}\,2\Theta(HD_{k-1})^{2}\,\mu(LD_{k})\leq 2 \delta_{0}^{2}\sigma(HD_{k-1})
$$
and using the fact that all the $BR$ cubes in our  tree $\wt \TT$ have little mass, we obtain
$$
\sigma(BR_{k}) \leq 4B^{2}\Theta(HD_{k-1})^{2}\mu(BR_{k}) \leq 4\delta_{W}B^{2}\sigma(HD_{k-1}),
$$
assuming $\delta_{W}\ll (2B)^{-2}$. So $\sigma(BR_{k}) +\sigma(LD_{k})\ll\sigma(HD_{k-1})$.
If moreover we assume $\delta_{W}\ll (2B)^{-2}A^{-1}$ and $\delta_{0}\ll A^{-1}$, together with \rf{eqfa520} and \rf{e.hd1}, this implies that for all $k\geq 1$
\begin{equation}\label{eqqq234}
\sigma(BR_{k}) +\sigma(LD_{k})\leq \frac1{100} \sigma(HD_{k})
\end{equation}
assuming $A$ big enough, and again by \rf{eqfa520} and \rf{e.hd1}, one deduces 
\rf{e.dechd} and \rf{e.uglycontrol}. This concludes the proof.
\end{proof}


\vv
We are now ready to estimate $W$ and $LW$ trees.

\begin{lemma}\label{lemtw}
If $\TT$ is of type $W$ or $LW$, then
$$\sigma(\TT)\leq C(A,B,\delta_0,M,\delta_W)\, \sum_{Q\in\TT}\|D_Q(\RR \mu)\|_{L^2(\mu)}^2.$$
\end{lemma}

\begin{proof}
We prove the case of a $LW$ tree of order $N_{0}\geq 1$, as the case of a $W$ follows with $N_{0}=1$.

We claim that it is enough to prove the following estimate:
\begin{equation}\label{e.wond}
\sum_{R\in HD_{N_{0}-1}}\sum_{Q\in BR(R)}\sum_{Q\subsetneq P \subset R} \|D_P(\RR \mu)\|_{L^2(\mu)}^2 \geq C(A,B, \delta_0,M, \delta_{W}) \sigma(HD_{N_{0}-1}).
\end{equation}
In fact, from \eqref{e.wond} and Lemma \ref{l.sigmatinc}  we see that
\begin{align*}
\sum_{Q\in\TT}\|D_Q(\RR \mu)\|_{L^2(\mu)}^2 &\geq \sum_{R\in HD_{N_{0}-1}}\sum_{Q\in BR(R)}\sum_{Q\subsetneq P \subset R}  \|D_P(\RR \mu)\|_{L^2(\mu)}^2\\
& \geq  C(A, B,\delta_0, M, \delta_{W})\sigma(HD_{N_{0}-1}) \geq \widetilde{C}(A, B, M, \delta_{0}, \delta_{W}) \sigma(\TT),
\end{align*}
which gives exactly the desired estimate.

Let us proceed with the proof of \eqref{e.wond}: 
$$\sum_{R\in HD_{N_{0}-1}}\sum_{Q\in BR(R)}\sum_{Q\subsetneq P \subset R}  \|D_P(\RR \mu)\|_{L^2(\mu)}^2  = \sum_{R\in HD_{N_{0}-1}}\sum_{Q\in BR(R)} \|  \sum_{Q\subsetneq P \subset R} D_P(\RR \mu)\|_{L^2(\mu)}^2.$$
For $R,Q$ as in the last sum we have
$$ \Bigl|\chi_Q\sum_{Q\subsetneq P \subset R} D_P(\RR \mu)\Bigr|
= \bigl|m_{Q}(\RR \mu) - m_{R}(\RR \mu)\bigr|
\geq \frac M2\,\Theta(R),
$$
by Lemma \ref{lemt0} (b). So we get
\begin{align*}
\sum_{R\in HD_{N_{0}-1}}\sum_{Q\in BR(R)}\sum_{Q\subsetneq P \subset R}  \|D_P(\RR \mu)\|_{L^2(\mu)}^2 
& \geq \sum_{R\in HD_{N_{0}-1}}\sum_{Q\in BR(R)}  \frac{M^2}{4} \Theta(R)^2 \mu(Q)\\
& \gtrsim  \sum_{R\in HD_{N_{0}-1}} \delta_{W}\frac{M^2}{4} \Theta(HD_{N_{0}-1})^2 \mu(R) \\
&\gtrsim \delta_{W}\frac{M^2}{4} \sigma(HD_{N_{0}-1}),
\end{align*}
using the condition \eqref{e.stoplw} in the second inequality.
\end{proof}

\vv


\subsection{Initial estimates for $TInc$ trees}

 We want to prove the following.

\begin{lemma}\label{lemtame}
Let $\TT$ be a $TInc$  tree. Then
$$\sigma(\TT)\leq C(A,B,M,\delta_{0}, \delta_{W}) \sum_{Q\in\TT}\|D_Q(\RR \mu)\|_{L^2(\mu)}^2.$$
\end{lemma}

The proof requires a deep analysis that will be mostly carried in section 7. We prove some preliminary estimates below. Let us consider a $TInc$ tree $\TT$ of order $N_0$. By \eqref{e.dechd} it follows that
$$\sigma(HD_{N_0-1})\geq \frac{A}{2} \,\sigma(HD_{N_0-2})\qquad \mbox{for $N_0>2$.}.$$
Taking also \rf{e.uglycontrol} into account, identifying $HD_{N_0-2}$ with the root of $\TT$ when $N_0=2$, we deduce that
\begin{equation}\label{eqdsa28}
\sigma(HD_{N_0-1})\geq \frac{1}{2A} \,\sigma(HD_{N_0-2}) \qquad \mbox{for $N_0\geq2$.}\end{equation}
Using also \rf{eqqq234} and the stopping condition \eqref{e.tame} we get
\begin{equation}\label{eqdsa29}
\sigma(HD_{N_0-1})\geq \frac12\,\sigma(\sss(\wt\TT_{N_0-1})) \geq \frac1{2A}\,
\bigl (\sigma(HD_{N_0}) + \sigma(BR_{N_0}) + \sigma(LD_{N_0})\bigr).
\end{equation}

Let $\GG$ be the collection of those cubes  $R\in HD_{N_0-2}$ such that
\begin{equation}\label{eqdefg22}
\sum_{\substack{P\in HD_{N_0-1}:\\P\subset R}} \sigma(P)\geq 
\max\Biggl(\frac{1}{20A}\,\sigma(R)\;,\;\;
\frac1{10A} 
\sum_{\substack{L\in HD_{N_0}\cup BR_{N_0}\cup LD_{N_0}:\\L\subset R}} \sigma(L)\Biggr).
\end{equation}

\begin{lemma}
\label{l.redtolast}
Let $\TT$ be a $TInc$  tree and $\GG$ as above. We have
$$\sigma(\GG)\geq C(B)\,\sigma(HD_{N_0-1}).$$
\end{lemma}

\begin{proof}
Let $\BB = HD_{N_0-2}\setminus \GG$. Let $\BB_1$ the collection of cubes  $R\in HD_{N_0-2}$ such that
$$\frac{1}{20A}\,\sigma(R) > \sum_{\substack{P\in HD_{N_0-1}:\\P\subset R}} \sigma(P)$$
and set $\BB_2=\BB\setminus \BB_1$. Thus the cubes from $\BB_2$ belong to $HD_{N_0-2}$ 
and satisfy
$$\sum_{\substack{P\in HD_{N_0-1}:\\P\subset R}} \sigma(P) < \frac1{10A} 
\sum_{\substack{L\in HD_{N_0}\cup BR_{N_0}\cup LD_{N_0}:\\L\subset R}} \sigma(L).$$

From the definition of $\BB_1$ and \rf{eqdsa28} we get
$$\sum_{R\in\BB_1} \sum_{\substack{P\in HD_{N_0-1}:\\P\subset R}} \sigma(P)
\leq \frac1{20A}\,\sum_{R\in\BB_1}\sigma(R)\leq \frac 1{20A} \,\sigma(HD_{N_0-2})  
\leq \frac 1{10}\, \sigma(HD_{N_0-1}).
$$
On the other hand, from the definition of $\BB_2$ and \rf{eqdsa29},
\begin{align*}
\sum_{R\in\BB_2} \sum_{\substack{P\in HD_{N_0-1}:\\P\subset R}} \sigma(P)
& \leq \frac{1}{10A} \sum_{R\in\BB_2}\,\,
\sum_{\substack{L\in HD_{N_0}\cup BR_{N_0}\cup LD_{N_0}:\\P\subset R}} \sigma(L) \\
& \leq \frac{1}{10A} \,\bigl (\sigma(HD_{N_0}) + \sigma(BR_{N_0})+ \sigma(LD_{N_0})\bigr)
\leq \frac{1}{5}\,\sigma(HD_{N_0-1}).
\end{align*}
Therefore,
$$\sum_{R\in\BB} \sum_{\substack{P\in HD_{N_0-1}:\\P\subset R}} \sigma(P)
\leq \frac{3}{10}\,\sigma(HD_{N_0-1}),$$
and so
$$\sum_{R\in\GG} \sum_{\substack{P\in HD_{N_0-1}:\\P\subset R}} \sigma(P)
\geq \frac{7}{10}\,\sigma(HD_{N_0-1}).$$
Observe now that each $R\in HD_{N_0-2}$ (in particular each $R\in\GG$) satisfies
$$\sum_{\substack{P\in HD_{N_0-1}:\\P\subset R}} \sigma(P)\leq c\,B^2\,\sigma(R).$$
So we deduce
$$\sigma(\GG) =\sum_{R\in\GG} \sigma(R) \geq \frac1{B^2}\sum_{R\in\GG} \,\sum_{\substack{P\in HD_{N_0-1}:\\P\subset R}}\! \sigma(P) \geq \frac{7c^{-1}}{10B^2}\,\sigma(HD_{N_0-1}).$$
\end{proof}

\vv
For the proof of Lemma \ref{lemtame}, notice that by Lemmas \ref{l.sigmatinc} and \ref{l.redtolast} we have
\begin{align*}
\sigma(\TT) &\leq C(A,B)\,\sigma(HD_{N_0-1})\leq \wt C (A,B)\,\sigma(\GG).
\end{align*}
We will show below that
\begin{equation} \label{e.redtolast}
\sigma(\GG)\leq  C(A,B, M, \delta_{0}, \delta_{W}) \sum_{R\in \GG} \sum_{P\in\TT(R)} \|D_P(\RR \mu)\|_{L^2(\mu)}^2, 
\end{equation}
where $\TT(R)$ is the tree formed by the cubes from $\TT$ which are contained in $R$.

Observe that the proof of Lemma \ref{lemtame} is complete if we show \rf{e.redtolast}.  However the arguments involved are very delicate and so it will be
addressed in Section 7. We introduce some necessary notation next.

We are going to define a tractable tree. The reader may think of it as a $TInc$ tree of order $2$ that involves different constants in the stopping conditions.

\begin{definition}
\label{d.tract}
Let $R\in\ttt$ be a cube. We say that $\TT$ is a tractable tree with root $R$ if it is a collection 
of cubes of the form
\begin{equation*}
\TT:= \tree(R)\cup \bigcup_{P\in HD(R)} \tree(P)
\end{equation*}
satisfying the following conditions:
\begin{align} \label{e.1stop}
\frac{1}{20A}\sigma(R)\leq & \sigma(HD_{1}),\\ \label{e.2stop}
\sigma(HD_{2}\cup BR_{2} \cup LD_{2})\leq &10A\sigma(HD_{1}),\\ \label{e.sbr}
\sigma(BR_{1})\leq  \delta_{W} \sigma(R)\; \text{ and } &\;\sigma(BR_{2})\leq  \delta_{W} \sigma(HD_{1}),
\end{align}
where the notation above is analogous to the one for $TInc$ trees, that is, $HD_{1}=HD(R)$, $LD_{1}=LD(R)$ and $BR_1=BR(R)$, $HD_{2}= \bigcup_{P\in HD(R)} HD(P)$, $LD_{2}=\bigcup_{P\in HD(R)} LD(P)$ and $BR_{2}=\bigcup_{P\in HD(R)} BR(P)$, $\sss_{1}(\TT)=HD_{1}\cup LD_{1}\cup BR_1$,  and $\sss_{2}(\TT)=LD_{1}\cup BR_1\cup HD_{2}\cup LD_{2}\cup BR_2$.
\end{definition}

\begin{rem}
Tractable trees are essentially $TInc$ trees of order 2, in particular they satisfy
\begin{align} \label{e.is}
\sigma(\sss_{1}(\TT)) & \gtrsim A^{-1}\sigma(R)\\ \label{e.ss}
\sigma(\sss_{2}(\TT)) & \lesssim A\sigma(\sss_{1}(\TT)).
\end{align} 

Notice that given a $TInc$ tree $\wt \TT$ and  $Q\in \GG$, then $\TT(Q)$ is a tractable tree. We will study tractable trees in Section 7.  

\end{rem}


\vv
\section{Some Fourier calculus and a maximum principle}

In this section we address some auxiliary questions which will be needed for the study of the tractable trees in the next section.

The Fourier transform of the kernel $\dfrac{x}{|x|^{s+1}}$ is
$c\,\dfrac\xi{|\xi|^{d-s+1}}$. Thus, if $\vphi$ is a $\CC^\infty$ function which is compactly supported in $\R^d$, the function $\psi:\R^d\to\R^d$ whose Fourier transform is
$$\wh \psi(\xi) = c\,|\xi|^{d-s-1}\,\xi\,\wh\vphi(\xi)$$
satisfies
$\RR(\psi\,d\LL^d)=\vphi$, where $c$ is some appropriate  
constant. Notice that
$$\RR(\psi\,\LL^d)(x)=\int \frac{x-y}{|x-y|^{s+1}}\cdot \psi(y)\,d\LL^d(y),$$
where the dot stands for the scalar product in $\R^d$. 

One can easily check that, although $\psi$ is not compactly supported, it is quite well localized, in the sense
that it satisfies
\begin{equation}\label{equas1}
|\psi(x)|\leq \frac{C_{\vphi,\alpha}}{1+|x|^{d+\alpha}},
\end{equation}
for $0\leq\alpha<d-s$.

In the next section we will need to consider a $C^\infty$ function $\vphi_0$ 
which equals $1$ on the ball $B(0,1)$ and vanishes out of $B(0,1.1)$, say.
Given a cube $Q\in\DD$ and some point $z_Q\in Q$ that will be fixed below
we denote $\vphi_Q = \vphi_0\circ T_Q$, where $T_Q$ is the affine map that
maps $B(0,1)$ to $B(x_Q,\frac1{10}c_{sep}\ell(Q))$. We also denote by $\psi_Q$ a vectorial function such that 
$\RR(\psi_Q\,\LL^d)=\vphi_Q$. That is,
$$\wh \psi_Q(\xi) = c\,|\xi|^{d-s-1}\,\xi\,\wh\vphi_Q(\xi).$$
It is straightforward to check that $\psi_Q = \ell(Q)^{s-d}\,\psi_0\circ T_Q$,
where $\psi_0$ is given by $\RR(\psi_0\,\LL^d)=\vphi_0$. 
Then it follows that
\begin{equation}\label{equas2}
\|\psi_Q\|_1= \ell(Q)^{s-d}\,\|\psi_0\circ T_Q\|_1 = C_{1}\,\ell(Q)^s,
\end{equation}
where $C_{1}=\|\psi_0\|_1$.  Also from \rf{equas1},
\begin{equation}\label{equas10}
|\psi_Q(x)|\leq \frac{C_{\vphi,\alpha}\,\ell(Q)^{s-d}}{1+\ell(Q)^{-d-\alpha}|x-z_Q|^{d+\alpha}} = 
\frac{C_{\vphi,\alpha}\,\ell(Q)^{s+\alpha}}{\ell(Q)^{d+\alpha}+|x-z_Q|^{d+\alpha}}
\end{equation}
for $0\leq\alpha<d-s$.

\begin{lemma}\label{lemmaxprin}
Let $\nu$ be a signed compactly supported Borel measure which is absolute continuous with respect to
Lebesgue measure, and suppose that its density is an $L^\infty$ function. Let $h:\R^d\to\R^d$ be an $L^\infty$
vector field. Let $1<r<\infty$ and $a>0$. If
$$|\RR\nu(x)|^r  + \RR^*(h\,\nu)(x) \leq a\qquad\mbox{for all $x\in\supp\nu$},$$
then the same inequality hods in all $\R^d$. That is,
\begin{equation}\label{eqaaa16}
|\RR\nu(x)|^r  + \RR^*(h\,\nu)(x) \leq a\qquad\mbox{for all $x\in\R^d$},
\end{equation}
where $\RR^*$ is the operator dual to $\RR$.
\end{lemma}

Almost the same arguments as the ones in \cite[Section 17]{ENV12} can be applied to prove this lemma. For the sake of completeness we give the detailed proof below.

\begin{proof}
We will use the fact that for any $L^\infty$ vector field $f$, the maximum of $\RR^*(f\,\nu)$ is attained
in $\supp \nu$ if $\sup_{x\in\R^d}\RR^*(f\,\nu)>0$. 
The fact that the maximum exists in this case is due to the fact that $\RR^*(f\,\nu)$ is
continuous in $\R^d$ and vanishes at $\infty$. On the other hand, in \cite{ENV12} it is shown
that the maximum is attained on $\supp\nu$ if $f\,\nu$ has a $\CC^\infty$ density with respect to the Lebesgue measure. However the same holds if the density is just $L^\infty$ and compactly supported. This can be deduced from a limiting argument that we omit.

Let us turn our attention to $|\RR\nu|^r  + \RR^*(h\,\nu)$ now. Again this is a continuous function
vanishing at $\infty$ and thus the maximum is attained if its supremum in $\R^d$ is positive.
 To show that this is attained in $\supp \nu$
we linearize the problem in the following way.
First notice that for any $y\in\R^d$ and $1<r<\infty$,
$$\frac1r \,|y|^r = \sup_{\lambda\geq0,|e|=1} \lambda\,\langle e,y\rangle - \frac1{r'}\,\lambda^{r'},$$
where $r'=r/(r-1)$. This follows either from elementary methods or from the fact that 
$\frac1{r'}\,\lambda^{r'}$ is the Legendre transform of $F(t)=\frac1r \,t^r$. Then we have
\begin{align*}
|\RR\nu(x)|^r  + \RR^*(h\,\nu)(x)& = \sup_{\lambda\geq0,|e|=1} r\,\lambda\,\langle e,\RR\nu(x)\rangle - \frac r{r'}\,\lambda^{r'} +  \RR^*(h\,\nu)(x)\\ & 
= \sup_{\lambda\geq0,|e|=1}  -\RR^*(r\,\lambda\,e\,\nu)(x) - \frac r{r'}\,\lambda^{r'} +  \RR^*(h\,\nu)(x)\\
& = \sup_{\lambda\geq0,|e|=1}  \RR^*(h\,\nu -r\,\lambda\,e\,\nu)(x) - \frac r{r'}\,\lambda^{r'}. 
\end{align*}
As mentioned above, for each $\lambda>0$ and each $e$ with $|e|=1$, the maximum of 
$\RR^*(h\,\nu -r\,\lambda\,e\,\nu)$ is attained in $\supp\nu$ if $\sup_{x\in\R^d}\RR^*(h\,\nu -r\,\lambda\,e\,\nu)(x)>0$. Thus the maximum of $|\RR\nu|^r  + \RR^*(h\,\nu)$ is attained in $\supp\nu$ whenever
\begin{equation}\label{eqaaa17}
\sup_{x\in\R^d}|\RR\nu(x)|^r  + \RR^*(h\,\nu)(x)>0.
\end{equation} 
The lemma is an immediate consequence of this statement. Indeed, notice that when proving \rf{eqaaa16} one can assume that \rf{eqaaa17} holds,  since $a>0$.
\end{proof}


\vv
\section{The tractable trees}

This section is devoted to the proof of the key estimate \eqref{e.redtolast}. In fact, this is a consequence of the following lemma.

\begin{lemma} \label{l.tract}\label{lemtrac}
Let $\TT$ be a tractable tree and $R$ its root. Then
$$\sigma(\{R\})\leq C(A,B,M,\delta)\, \sum_{Q\in\TT}\|D_Q(\RR \mu)\|_{L^2(\mu)}^2.$$
\end{lemma}

Before turning to the arguments for the proof of the preceding result we show how 
Lemma  \ref{lemtame} follows from this.

\vv
\begin{proof}[\bf Proof of Lemma  \ref{lemtame} using Lemma \ref{l.tract}]
Recall that we have to prove that, if  $\TT$ be a $TInc$  tree, then
$$\sigma(\TT)\leq C(A,B,M,\delta_{0}, \delta_{W}') \sum_{Q\in\TT}\|D_Q(\RR \mu)\|_{L^2(\mu)}^2.$$

By Lemmas \ref{l.sigmatinc} and \ref{l.redtolast} we have
\begin{align*}
\sigma(\TT) &\leq C(A,B)\,\sigma(HD_{N_0-1})\leq C(A,B)\,\sigma(\GG).
\end{align*}
where $\GG$ is the family of cubes from $HD_{N_0-2}$ defined in \rf{eqdefg22}.
For $R\in\GG$, it turns out that $\TT(R)$, the family of cubes from $\TT$ contained in $R$, is
a tractable tree. Thus, by Lemma \ref{l.tract} we have
\begin{align*}
\sigma(\GG)= \sum_{R\in\GG}\sigma(R) 
&\leq C(A,B,M,\delta)\, \sum_{R\in\GG}\sum_{Q\in\TT(R)}\|D_Q(\RR \mu)\|_{L^2(\mu)}^2\\
&\leq C(A,B,M,\delta)\, \sum_{Q\in\TT}\|D_Q(\RR \mu)\|_{L^2(\mu)}^2.
\end{align*}
\end{proof}

\vv

We introduce now some additional notation. As in section 5, we use the notation $\Theta(HD_{k})$ for the dyadic density $\Theta_{d}(Q)$ constant for all cubes $Q\in HD_{k}$, $k\in \mathbb N$. We consider the following approximating measure for $\mu\rest R$:
\begin{equation}
\label{e.approxm}
\eta = \sum_{P\in\sss_2(\TT)} \frac{\mu(P)}{\LL^d(P)}\,
\LL^d\rest P.
\end{equation}
Recall that, by assumption, $\LL^d(P)\approx\ell(P)^d$.
In order to compare Riesz transforms with respect to $\mu$ and $\eta$, it is convenient to introduce some
coefficients $q(\cdot)$: for $Q\in\TT$, we set
$$q(Q,\TT)=\sum_{P\in\sss_{2}(\TT)} \frac{\ell(P)\,\mu(P)}{D(P,Q)^{s+1}},$$
where
$$D(P,Q) = \ell(P)+\dist(P,Q) + \ell(Q).$$
Notice that the coefficients $q(Q,\TT)$ depend on $Q$ and on the cubes from $\sss_{2}(\TT)$.

\begin{lemma}\label{lemdifer}
For every $Q\in\TT$ and $x,x'\in Q$,
$$|\RR (\chi_{R\setminus Q} \,\mu)(x) - \RR (\chi_{R\setminus Q} \,\eta)(x')|\lesssim p(Q,R)+ q(Q,\TT).$$
\end{lemma}

\begin{proof}
For each $P\in\sss_{2}(\TT)$ let $z_P\in P$ be some fixed point. Taking into account that $\mu(P)=\eta(P)$ for
each such cube, for $x\in Q$, we get
\begin{align*}
\RR (\chi_{R\setminus Q} \,\mu)(x) - \RR &(\chi_{R\setminus Q} \,\eta)(x) \\
& =
\sum_{\substack{P\in\sss_{2}(\TT):\,P\not\subset Q}}\left( \int_P K^s(x-y)\,d\mu(y) - 
 \int_P K^s(x-y)\,d\eta(y)\right) \\
&=\sum_{\substack{P\in\sss_{2}(\TT):\, P\not\subset Q}}
\int_P \bigl[K^s(x-y) - K^s(x-z_P)\bigr]\,d\mu(y)\\ 
&\quad  -
\sum_{\substack{P\in\sss_{2}(\TT):\, P\not\subset Q}}\int_P \bigl[K^s(x-y) - K^s(x-z_P)\bigr]\,d\eta(y).
\end{align*}
For $x$ and $y$ as in the preceding integrals we have
$$\bigl|K^s(x-y) - K^s(x-z_P)\bigr|\lesssim \frac{\ell(P)}{D(P,Q)^{s+1}}.$$
Then we deduce that
$$|\RR (\chi_{R\setminus Q} \,\mu)(x) - \RR (\chi_{R\setminus Q} \,\eta)(x)|\lesssim q(Q,\TT).$$
The lemma follows by combining this estimate with the fact that for $x,x'\in P$,
$$|\RR (\chi_{R\setminus Q} \,\mu)(x) - \RR (\chi_{R\setminus Q} \,\mu)(x')|\lesssim p(Q,R).$$
\end{proof}

\vv

\begin{lemma}\label{lemdifer22}
For $1\leq r<\infty$, we have
$$\sum_{P\in\sss_{2}(\TT)} q(P,\TT)^r \mu(P)\lesssim c_{6}\,\sum_{P\in\sss_{2}(\TT)}p(P,R)^r\,\mu(P).$$
\end{lemma}

\begin{proof}
Let us consider the functions $p$ and $q$ defined on $R$ by $p(x):=p(P_{x})$ and $q(x):=q(P_{x}, \TT)$, where $P_{x}$ is the cube in $\sss_{2}(\TT)$ that contains $x$. We want to prove $\|q\|_{L^{r}}\leq c \|p\|_{L^{r}}$. Let us fix $\lambda>0$ and consider the set $\Omega_{\lambda}:=\{x\in \R^d: \, q(x)>\lambda \}$. Take a Whitney decomposition of $\Omega_{\lambda}$ into a family of ``true cubes'' $W(\Omega_{\lambda})$. In particular, we assume that the cubes $L\in W(\Omega_{\lambda})$ satisfy
$$\dist(3L,\R^d\setminus \Omega_\lambda)\approx\ell(L).$$

 We claim that there exists an absolute constant $C>1$ such that
\begin{equation} \label{e.goodl}
\mu(\{x\in \R^d: \,\, q(x)>C\lambda \}) \leq \sum_{L\in W(\Omega_{\lambda})} \mu\Bigl(\Bigl\{x\in L: \,\, \sum_{\substack{Q\in \sss_{2}(\TT)\\  Q\subset 3L}}\frac{\mu(Q)\ell(Q)}{D(P_{x},Q)^{s+1}}>\lambda \Bigl\}\Bigr).
\end{equation}
We first show that the lemma follows from \eqref{e.goodl}. By Chebychev, we deduce
\begin{align*}
\mu(\{x\in \R^d: \,\, q(x)>C\lambda \}) & \leq \sum_{L\in W(\Omega_{\lambda})} \frac{1}{\lambda} \int_{L}  \sum_{\substack{Q\in \sss_{2}(\TT)\\  Q\subset 3L}} \frac{\mu(Q)\ell(Q)}{D(P_{x},Q)^{s+1}} d\mu(x)\\
&= \sum_{L\in W(\Omega_{\lambda})} \frac{1}{\lambda} \int_{L}\int_{3L} \frac{\ell(P_{y})}{D(P_{x},P_{y})^{s+1}}d\mu(y) d\mu(x).
\end{align*}
By Fubini, the last double integral is bounded by
\begin{align*}
\int_{3L} &\ell(P_{y})\int_{L} \frac{d\mu(x)}{\left( \ell(P_{y})+|x-y|\right)^{s+1}}\,d\mu(y)\\
&\lesssim  \int_{3L} \ell(P_{y}) \sum_{\substack{Q\in \mathcal D\\ Q\supseteq P_{y}}}\left(\int_{ Q^1\setminus Q}  \frac{d\mu(x)}{\left(\ell(P_{y})+|x-y| \right)^{s+1}}+ \int_{P_{y}}  \frac{d\mu(x)}{\left(\ell(P_{y})+|x-y| \right)^{s+1}}\right)\,d\mu(y)\\
& \lesssim \int_{3L} \sum_{\substack{Q\in \mathcal D\\ Q\supseteq P_{y}}} \frac{\ell(P_{y})\mu(Q)}{\ell(Q)^{s+1}}\,d\mu(y) \lesssim\int_{3L} p(y)d\mu(y),
\end{align*}
where $Q^1$ stands for the parent cube of $Q$.
Thus we get
$$\mu(\{x\in \R^d: \,\, q(x)>C\lambda \}) 
\lesssim \frac{1}{\lambda} \int_{\Omega_{\lambda}}p(y)d\mu(y).$$
We use this estimate to get the desired result:
\begin{align*}
\|q\|_{L_{r}}^{r} & \lesssim \int_{0}^{\infty} \lambda^{r-1} \mu(\{x\in \R^d: \,\, q(x)>C\lambda \}) d\lambda\\
 & \lesssim \int_{0}^{\infty} \lambda^{r-2} \int_{\Omega_{\lambda}}p(x) d\mu(x) d\lambda\\
 & = \int p(x) \int_{0}^{q(x)} \lambda^{r-2}d\lambda d \mu(x)\\
  & \approx \int p(x) q(x)^{r-1}d\mu(x)\\
  &\leq \left( \int p(x)^{r}d\mu(x)\right)^{1/r}\left( \int q(x)^{r'(r-1)} d\mu(x) \right)^{1/r'}= \|p\|_{L_{r}}\|q\|_{L_{r}}^{r/r'},
\end{align*}
where $r'$ is the dual exponent of $r$. So we obtain $\|q\|_{L_{r}}\lesssim \|p\|_{L_{r}}$, as wished. 

We are only left with the proof of \eqref{e.goodl}.  For $x\in \Omega_{\lambda}$, let $L \in W(\Omega_{\lambda})$ be such that $x\in L$. Take $x'\in \partial \Omega_{\lambda}$
such that $|x-x'|\approx \ell(L)$. As $x'\not\in\Omega_\lambda$,
\begin{equation*} 
q(x')=\sum_{Q\in \sss_{2}\TT} \frac{ \mu(Q)\ell(Q)}{D(P_{x'}, Q)^{s+1}}\leq \lambda.
\end{equation*}
We split $q(x)$ in two summands, the local part and the non-local one: 
$$
q(x)= \sum_{\substack{Q\in \sss_{2}\TT\\ Q\subset 3L}} \frac{ \mu(Q)\ell(Q)}{D(P_{x}, Q)^{s+1}} + \sum_{\substack{Q\in \sss_{2}\TT\\ Q\not\subset 3L}} \frac{ \mu(Q)\ell(Q)}{D(P_{x}, Q)^{s+1}}.
$$
Observe that the local part is the one that appears on the right hand side of \rf{e.goodl}.
Thus to prove the claim it is enough to show that 
\begin{equation} \label{e.lclaim}
 \sum_{\substack{Q\in \sss_{2}\TT\\ Q\not\subset 3L}} \frac{ \mu(Q)\ell(Q)}{D(P_{x}, Q)^{s+1}}\leq  C_{nl}\sum_{\substack{Q\in \sss_{2}\TT\\ Q\not\subset 3L}} \frac{ \mu(Q)\ell(Q)}{D(P_{x'}, Q)^{s+1}}, 
\end{equation}
where $C_{nl}$ is some constant depending at most on $d$, $s$, and $c_{sep}$. We will then take $C:= C_{nl} +1$ in \rf{e.goodl}.

To prove \rf{e.lclaim},
observe that the numerators are the same in both sides. So it suffices to show that 
\begin{equation}\label{eqdd392}
D(P_{x'},Q)\leq C\, D(P_{x},Q) \quad \mbox{ when $Q\not\subset 3L$.}
\end{equation}
 First notice that, by the separation condition \eqref{eqsepara}, if $P_{x'}\neq Q$, then 
 $\ell(P_{x'})\leq c\dist(P_{x'},Q)$. If $P_{x'}= Q$, then obviously $\ell(P_{x'})= \ell(Q)$.
 So in any case
 $$\ell(P_{x'})\leq \ell(Q) + c\dist(P_{x'},Q).$$
Then to obtain \rf{eqdd392}, it is enough to prove 
\begin{equation}
\label{e.copyq}
\dist(P_{x'},Q)\leq c(\ell(P_{x})+\ell(Q)+ \dist(P_{x},Q))
\end{equation}
Since
$$\dist(P_{x'},Q)\leq \dist(P_{x'},P_{x}) + \dist(P_{x},Q)+ \ell(P_x),$$
it suffices to see that
$\dist(P_{x'},P_x)\leq c(\ell(P_{x})+\ell(Q)+ \dist(P_{x},Q))$.
This follows from
the Whitney condition and the fact that $Q\not\subset 3L$:
$$
\dist(P_{x'},P_{x})\leq |x-x'|\lesssim \ell(L) \lesssim \dist(P_{x},Q) + \ell(Q) + \ell(P_x).
$$
So holds \eqref{e.copyq} holds. This completes the proof of  \eqref{e.lclaim} and the claim \eqref{e.goodl}.
\end{proof}

Observe that the estimates obtained in Lemmas \ref{lemdifer} and \ref{lemdifer22} do not depend on the constants $A$, $B$, $M$, $\delta_0$ or $\delta_W$. In fact, they do not depend 
on the precise construction of the tree $\TT$.

\vv

\vv

\begin{lemma}\label{lemboundeta}
For every $k\geq0$, let $HD_1^k$ be the collection of cubes $Q\in\TT$ which satisfy
$\Theta(Q)\geq 2^k\Theta(HD_1)$ and have maximal side length. Denote
$$H_k=\bigcup_{Q\in HD_1^{k}} Q \qquad \text{and}\qquad \wt H_k=\bigcup_{Q\in HD_1^{k}} 
(1+\tfrac1{10}c_{sep})\,Q, $$
and let $\Phi(x)=\dist(x, \wt  H_k ^{c})$. Set $\Theta(HD_1^k) = 2^k\Theta(HD_1)$. Then, for $1<r<\infty$,
$\RR_{\Phi,\eta}$ is bounded in $L^r(\eta)$ with norm not exceeding $c(r,M)\,\Theta(HD_1^k)$.
\end{lemma}

\begin{proof}
For $1<r<\infty$, the boundedness of $\RR_{\Phi,\eta}$ in $L^r(\eta)$ with norm $\lesssim \Theta(HD_1^k)$
follows from the boundedness in $L^2(\eta)$. This is shown in \cite{Vo}
(see also Lemma 5.27 of \cite{To}). So we only have to deal with the case $r=2$.

Denote
$$F = \wt H_k\cup \bigcup_{P\in\sss_2(\TT)}(1+\tfrac1{10}c_{sep})\,P,$$
and consider the auxiliary function 
$$\Psi(x)= \dist(x,F^c).$$
Notice that all the cubes $Q\in\DD$ such that $\Theta(Q)\geq 2^k\Theta(HD_1)$ are contained in $F$.
Then,
arguing as in Lemma \ref{lemrieszmax}, 
it follows that $\RR_{\Psi,\mu}$ is bounded in $L^2(\mu\rest R)$
with norm bounded by $C\Theta(HD_1^k)$, by an application of Theorem \ref{teontv}.

By approximation, we are going to show that $\RR_{\Psi,\eta}$ is bounded in $L^2(\eta)$, with its norm
not exceeding $C\Theta(HD_1^k)$. The first step consists in showing that 
$\RR_{\Psi,\eta}$ is bounded from $L^2(\eta)$ to $L^2(\mu\rest R)$ with norm not exceeding $c\,\Theta(HD_1^k)$.	 
To this end, we consider the auxiliary tree $\TT''\subset\TT$ whose stopping cubes are given by the family of maximal
cubes from $HD_1^{k}\cup\sss_{2}(\TT)$, which we denote by $\sss(\TT'')$. That is, $\TT''$ is obtained from $\TT$ 
after removing all the cubes from $\DD$ contained in $HD_1^k$. 
Given $f\in L^2(\eta)$, we consider the function
$\wt f\in L^2(\mu\rest R)$ which vanishes out of $R$ and is constantly equal to $\int_{P}f d\eta/\mu(P)$ on each $P\in
\sss(\TT'')$. So notice that 
$$\int_P f\,d\eta = \int_P \wt f\,d\mu\qquad\mbox{for all $P\in\sss(\TT'').$}$$
Since $\|\wt f\|_{L^2(\mu)}\leq \|f\|_{L^2(\eta)}$, to prove the boundedness of $\RR_{\Psi,\eta}$ from $L^2(\eta)$ to $L^2(\mu\rest R)$ with norm not exceeding $c\,\Theta(HD_1^k)$, it is enough to show that
\begin{equation}\label{eqdah3}
\int_R |\RR_\Psi(f\,\eta) - \RR_\Psi(\wt f\,\mu)|^2\,d\mu\leq C\,\Theta(HD_1^k)^2\,\|f\|_{L^2(\eta)}^2.
\end{equation}
Now we operate as in Lemma \ref{lemdifer}, and then for each $x\in Q\in\sss(\TT'')$, taking into account that $\int_P f\,d\eta= \int_P f\,d\mu$, we get
\begin{align*}
\RR_\Psi(f\,\eta)(x) - \RR_\Psi(\wt f\,\mu)(x) &
 =
\sum_{P\in\sss(\TT'')}\left( \int_P K_\Psi(x,y)\bigl[f(y)\,d\eta(y) - \wt f(y)\,d\mu(y)\bigr]\right) \\
&=\sum_{P\in\sss_{2}(\TT'')}
\int_P \bigl[K_\Psi(x,y) - K_\Psi(x,z_P)\bigr]f(y)\,d\eta(y)\\ 
&\quad  - 
\sum_{P\in\sss_{2}(\TT'')}
\int_P \bigl[K_\Psi(x,y) - K_\Psi(x,z_P)\bigr]\wt f(y)\,d\mu(y).
\end{align*}
Since $\Psi(y)\gtrsim\ell(P)$ for every $y\in P$, we have
$$\bigl|K_\Psi(x,y) - K_\Psi(x,z_P)\bigr|\lesssim \frac{\ell(P)}{|x-z_P|^{s+1} + \ell(P)^{s+1}}.$$
Then we deduce that
\begin{align*}
|\RR_\Psi(f\,\eta)(x) - & \RR_\Psi(\wt f\,\mu)(x)| \\
& \lesssim \sum_{P\in\sss_{2}(\TT'')}
\int_P \frac{\ell(P)}{|x-z_P|^{s+1} + \ell(P)^{s+1}}\,\bigl[|f(y)|\,d\eta(y) + |\wt f(y)|\,d\mu(y)\bigr]\\
& \lesssim \sum_{P\in\sss_{2}(\TT'')}
\int_P \frac{\ell(P)}{|x-z_P|^{s+1} + \ell(P)^{s+1}}\, |f(y)|\,d\eta(y).
\end{align*}
To estimate the last integral we argue by duality. Given $g\in L^{p'}(\eta)$,
\begin{align}\label{eqdafgq1}
\int & \sum_{P\in\sss_{2}(\TT)}
\int_P \frac{\ell(P)}{|x-z_P|^{s+1} + \ell(P)^{s+1}}\, |f(y)|\,d\eta(y)\,|g(x)|\,d\eta(x) \\
& = 
\sum_{P\in\sss_{2}(\TT)}
\int_P |f(y)|\left(\int \frac{\ell(P)}{|x-z_P|^{s+1} + \ell(P)^{s+1}}\, |g(x)|d\eta(x)\right)d\eta(y).
\nonumber
\end{align}
Using the fact that $\eta(\lambda P)\leq \mu(\lambda P \cap R)\leq C\,\Theta(HD_1^k)\,\ell(\lambda P)^s$ for every $\lambda\geq 1$, by splitting the domain of integration into annuli we infer that
\begin{equation}\label{eqmumu2}
\int \frac{\ell(P)}{|x-z_P|^{s+1} + \ell(P)^{s+1}}\, |g(x)|d\eta(x)\lesssim\Theta(HD_1^k)\,
\inf_{z\in P} M_\eta g(z),
\end{equation}
where $M_\eta$ stands for the centered Hardy-Littlewood operator with respect to $\eta$. So the left
side of \rf{eqdafgq1} is bounded by
$$C\,\Theta(HD_1^k)\sum_{P\in\sss_{2}(\TT)}
\int_P | f(y)|\, M_\eta g(y)\,d\eta(y) \lesssim \Theta(HD_1^k)\, \|f\|_{L^2(\eta)}\,\|g\|_{L^2(\eta)}.$$
Thus
\rf{eqdah3} is proved.

By duality we deduce now that $\RR_{\Psi,\mu}$ is bounded from $L^2(\mu)$ to $L^2(\eta)$, with norm not 
exceeding $C\,\Theta(HD_1^k)$. To prove the boundedness of $\RR_{\Psi,\eta}$ in $L^2(\eta)$, we consider again
$f$ and $\wt f$ as above. 
We claim that
$$
\int_R |\RR_\Psi(f\,\eta) - \RR_\Psi(\wt f\,\mu)|^2\,d\eta\leq C\,\Theta(HD_1^k)^2\,\|f\|_{L^2(\eta)}^2.
$$
This  follows by almost the same arguments as the ones above for the proof of \rf{eqdah3}. The details are left for
the reader.

Now we turn our attention to the operator $\RR_\Phi$. Given $x\in\supp\eta$, notice that $\Phi(x)\neq\Psi(x)$ only for
the points $x\in\sss_2(\TT)\setminus H_k$. Indeed, for such points $\Phi(x)=0$ while
$\Psi(x)=\ell(P_x)$, where $P_x$ is the cube from $\sss_{2}(\TT)$ which contains $x$.
Now we just write$$
|\RR_{\eta}f(x)|\leq |\RR_{c_{sep}\ell(P_x)/10,\eta}f(x)| + \int_{|x-y|\leq c_{sep}\ell(P_{x})/10}\frac{1}{|x-y|^{s}}|f(y)|d\eta(y).
$$
The first term on the right is controlled by $\RR_\Psi f(x)$ and an error in terms of $\Theta(HD_1^k)$, using \eqref{e.compsup}. The last one also can be easily controlled since the separation condition \eqref{eqsepara} gives that all $y$ such that $|x-y|\leq \frac1{10}\,c_{sep}\ell(P_{x})$ satisfies $y\in P_{x}$. On $P_{x}$ the measure $\eta$ is uniformly distributed with respect to Lebesgue and one can use standard analysis to controll it by $M_{\eta}f(x)\Theta(P_{x})\leq B^{2} \Theta(R)\, M_{\eta}f(x)$.
\end{proof}
\vv

\begin{lemma}
\label{l.t1eta}
Let $\TT$ be a tractable tree with roor $R$ and $\eta$ be the measure defined in \eqref{e.approxm}. Then for $1<r<\infty$ $\RR_{\eta}$ is bounded in $L^{r}(\eta)$ with norm not exceeding $C_2(B, M)\Theta(R)$. 
\end{lemma}

\begin{proof}
This result can be considered just as a particular case of Lemma \ref{lemboundeta}.
Indeed, take $k\in\Z$ such that $2B< 2^k\leq4B$. Then it follows that the set $H_k$ in that lemma is empty, and so $\Phi(x)=0$. Thus $\RR_\eta$ is bounded in $L^r(\eta)$ with norm not exceeding
$c(M,r)\,\Theta(HD_1^k)\lesssim C_2(M,B)\,\Theta(R)$.
\end{proof}

\vv

\begin{lemma}\label{lemvar1}
Let $\TT$ be as above and set 
\begin{equation}\label{eqdeff}
f(x) = \RR(\chi_{R^c}\mu)(x) - m_R(\RR \mu).
\end{equation}
For every $1<r<\infty$ there exists some constant $c_r>0$ such that
$$\|\RR\eta +f\|_{L^r(\eta)}^r \geq c_r\,\Theta(HD_1)^r\,\eta(HD_1).$$
\end{lemma}

\begin{proof}
For a collection of cubes $\AZ\subset\DD$ and $1<r<\infty$, we denote
$$\sigma_r(\AZ) = \sum_{P\in\AZ} \Theta(P)^r\,\mu(P),$$
so that $\sigma(\AZ)\equiv \sigma_2(\AZ)$.

For every $k\geq0$, recall $HD_1^k$ is the collection of cubes $Q\in\TT$ which satisfy
$\Theta(Q)\geq 2^k\Theta(HD_1)$ and have maximal side length (this ensures that $\Theta(Q)\approx 2^k\Theta(HD_1)$) and 
$H_k=\bigcup_{Q\in HD_1^{k}} Q.$ Observe that $H_{k+1}\subset H_k$.

Let $k_0\geq1$ be such that $\sigma_r(HD_1^{k_0})$ is maximal. For $l\geq1$, we have
$$2^{(k_0+l)r} \Theta(HD_{1})^r \eta(H_{k_0+l}) \approx 
\sigma_r(HD_1^{k_0+l})\leq \sigma_r(HD_1^{k_0}) \approx 2^{k_0r} \Theta(HD_{1})^r \eta(H_{k_0}).$$
Thus there exists some fixed $l$ (absolute constant) such that
$$\eta(H_{k_0+l})\leq \frac12\,\eta(H_{k_0}).$$
We consider the set
$$H= H_{k_0}\setminus H_{k_0+l},$$
so that $\eta(H)\approx\eta(H_{k_0})$. We also denote $\Theta(H)=2^{k_0}\Theta(HD_1)$.
Notice that
\begin{equation}\label{eqdif422}
\sigma_r(HD_1)\leq \sigma_r(HD_1^{k_0}) \approx \Theta(H)^r\,\eta(H).
\end{equation}

Let $z_Q\in Q$ be such that 
$$\eta\bigl(B(z_Q,\tfrac1{10}c_{sep}\ell(Q))\bigr)\approx \eta(Q),$$
with the comparability constant depending on $c_{sep}$. Consider the functions $\vphi_Q$ and
$\psi_Q$ defined just above \rf{equas1}. Observe that $\supp\vphi_Q\subset
\bar B(z_Q,0.11c_{sep}\ell(Q))$.
In the arguments below we will need to work with the functions
$$\vphi = \sum_{Q\in HD_1^{k_0}}\Theta(Q)\,\vphi_Q$$
and
$$\psi = \sum_{Q\in HD_1^{k_0}}\Theta(Q)\,\psi_Q,\qquad \wt \psi = \sum_{Q\in HD_1^{k_0}}\Theta(Q)\,|\psi_Q|,$$
where $\vphi_Q$ and $\psi_Q$ have been defined in the preceding section.
$\RR(\psi\,d\LL^d)=\vphi$ and, by \rf{equas2}, 
\begin{equation}\label{eqnormpsi}
\|\psi\|_1\leq\|\wt \psi\|_1\leq\!\sum_{Q\in HD_1^{k_0}}\Theta(Q)\,\|\psi_Q\|_1 = C_1\!\sum_{Q\in HD_1^{k_0}}\Theta(Q)\,
\ell(Q)^s = C_1\,\eta(HD_1^{k_0})\approx \eta(H).
\end{equation}
Recalling \rf{eqdif422}, to prove the lemma it suffices to show that
$$\|\RR\eta +f\|_{L^r(\eta)}^r \geq c_r\,\Theta(H)^r\,\eta(H).$$
We will argue
by contradiction, following some ideas inspired by the techniques from \cite{ENV12}
and \cite{NToV}. So suppose that 
\begin{equation}\label{eqas12}
\|\RR\eta +f\|_{L^r(\eta)}^r \leq \lambda\,\Theta(H)^r\,\eta(H),
\end{equation}
where $\lambda>0$ is some small constant to be fixed below. Consider the measures of the form
$\nu=a\,\eta$, with $a\in L^\infty(\eta)$, $a\geq 0$, and let $F$ be the functional
\begin{equation}\label{e.functional}
F(\nu) = \int |\RR\nu + f|^r\,d\nu + \lambda\,\|a\|_{L^\infty(\eta)} \,\Theta(H)^r\,\eta(H).
\end{equation}
Let
$$m= \inf F(\nu),$$ where the infimum is taken over all the measures $\nu=a\,\eta$, with $a\in L^\infty(\eta)$,
such that $\nu(H) = \eta(H)$. We call such measures admissible.
Note that $\eta$ is admissible, and thus
\begin{equation}
\label{e.admis}
m \leq F(\eta) \leq 2\lambda\,\Theta(H)^r\,\eta(H),
\end{equation}
by the assumption \rf{eqas12}. This tells us that the infimum $m$ is attained over the collection of
measures $\nu=a\eta$ with $\|a\|_{L^\infty(\eta)}\leq 3$. In particular, by taking weak $*$ limits in $L^\infty(\eta)$, this guaranties
the existence of a minimizer. 

Let $\nu$ be an admissible measure such that $m=F(\nu)$. We claim that
\begin{equation}\label{eqclam43}
|\RR\nu(x) + f(x)|^r + r \,\RR^*\bigl ((\RR\nu + f) |\RR\nu + f|^{r-2} \nu\bigr)(x) \leq C\,\lambda^{1/r'}\,\Theta(H)^r \quad \text{ on $\supp\nu$.}
\end{equation}
Let us assume this for the moment.
Observe that, by Lemma \ref{lemcomp}, for all $x\in\supp\nu\subset R$,
$$|f(x)| = |\RR(\chi_{R^c}\mu)(x) - m_R(\RR \mu)|\lesssim p(R) \lesssim \Theta(R) = B^{-1}\,\Theta(HD_1)
\leq B^{-1}\,\Theta(H).$$
Then, using the inequality $|\alpha + \beta|^r\leq 2^{r-1}
\bigl (|\alpha|^r + |\beta|^r\bigr)$, we infer that
\begin{align*}
|\RR\nu(x)|^r & + 2^{r-1} r \,\RR^*\bigl ((\RR\nu + f) |\RR\nu + f|^{r-2} \nu\bigr)(x) \\
& \leq 2^{r-1}\bigl (|\RR\nu(x) + f(x)|^r + |f(x)|^r\bigr) + 2^{r-1} r \,\RR^*\bigl ((\RR\nu + f) |\RR\nu + f|^{r-2} \nu\bigr)(x)\\
& \leq 
2^{r-1}|f(x)|^r + C\,\lambda\,\Theta(H)^r\\&\leq C \bigl (B^{-r} + \lambda^{1/r'}\bigr)\,\Theta(H)^r
\end{align*}
for $x\in\supp\nu$.
It is easily checked that the function on the left side of this inequality is continuous. Appealing then to the maximum principle in Lemma \ref{lemmaxprin},
with $h=2^{r-1} r (\RR\nu + f) |\RR\nu + f|^{r-2}$,
we infer the estimate above also holds out of $\supp\nu$. Therefore, integrating against $\wt \psi\,d\LL^d$ (recall that $\wt\psi$ is non-negative), we get
\begin{equation}\label{eqpss3}
\int \!|\RR\nu|^r \wt\psi\,d\LL^d + 2^{r-1} r\! \int \!\RR^*\bigl ((\RR\nu + f) |\RR\nu + f|^{r-2} \nu\bigr)\,
\wt\psi\,d\LL^d\leq C\bigl(B^{-r} + \lambda^{1/r'}\bigr)\,\Theta(H)^r\,\|\wt \psi\|_1.
\end{equation}

Next we estimate the second integral on the left side of the preceding inequality, which we denote by $I$.  For that purpose we will also need the estimate
\begin{equation}\label{eqacpsi}
\int |\RR(\wt \psi\,d\LL^d)|^r\,d\nu\lesssim \Theta(H)^r\,\eta(H),
\end{equation}
that we will prove later. Using \eqref{eqacpsi} we have
\begin{align*}
|I|& = \left|\int (\RR\nu + f) |\RR\nu + f|^{r-2} \,
\RR(\wt\psi\,d\LL^d)\,d\nu\right| \\
&\leq \left(\int |\RR\nu + f|^r \,d\nu\right)^{\frac{r-1}r}
\left(\int |\RR(\wt\psi\,d\LL^d)|^r\,d\nu\right)^{\frac{1}r}\\
& \lesssim \bigl (\lambda\,\Theta(H)^r\,\eta(H)\bigr)^{1/r'} \bigl (\Theta(H)^r\,\eta(H)\bigr)^{1/r} = \lambda^{1/r'}\,\Theta(H)^r\,\eta(H),
\end{align*}
where $r'=r/(r-1)$.
From \rf{eqpss3}, the preceding estimate, and the fact that $\|\wt\psi\|_1\lesssim\eta(H)$, we deduce
that
\begin{equation}\label{eqcont4}
\int |\RR\nu|^r \wt\psi\,d\LL^d \lesssim \bigl (B^{-r} + \lambda^{1/r'}\bigr)\,\Theta(H)^r\,
\eta(H).
\end{equation}

To get a contradiction, note that by the construction of $\psi$ and $\vphi$, we have
$$\left|\int \RR\nu\,\psi\,d\LL^d\right| = \left|\int \RR^*(\psi\,d\LL^d)\,d\nu \right|=  
\int \vphi\,d\nu =
\sum_{Q\in HD_1^{k_0}} \Theta(Q)\,\nu(Q) \approx \Theta(H)\,\eta(H).$$
On other hand, since $|\psi|\leq\wt\psi$, by H\"older's inequality, \rf{eqcont4}, and \rf{eqnormpsi},
$$\left|\int \RR\nu\,\psi\,d\LL^d\right| \leq \left(\int |\RR\nu|^r \wt\psi\,d\LL^d \right)^{1/r}
\left(\int \wt\psi\,d\LL^d\right)^{1/r'} \lesssim \bigl (B^{-r} +  \lambda^{1/r'}\bigr)^{1/r}\,\Theta(H)
\,\eta(H),$$
which contradicts the previous estimate if $B$ is big enough and $\lambda$ small enough. So to finish the proof of the lemma it only remains to prove the estimates \rf{eqacpsi} and \rf{eqclam43}. This task is carried out below.
\end{proof}

\vv
\begin{proof}[\bf Proof of \rf{eqacpsi}]
We have to show that
$$\int |\RR(\wt\psi\,d\LL^d)|^r\,d\nu\lesssim \Theta(H)^r\,\eta(H).
$$
Recall that 
$$\wt\psi = \sum_{Q\in HD_1^{k_0}}\Theta(Q)\,|\psi_Q|.$$
We consider now the function
$$g=\sum_{Q\in HD_1^{k_0}} g_Q,$$
where $g_Q = c_Q\,\chi_Q$, with $c_Q=\Theta(Q)\int|\psi_Q|\,d\LL^d /\eta(Q)$, so that
$$\int g_Q\,d\eta = \Theta(Q)\int |\psi_Q|\,d\LL^d\qquad\mbox{for all $Q\in HD_1^{k_0}.$}$$
Recalling that $\|\Theta(Q)\psi_Q\|_1\lesssim \eta(Q)$, it follows that $|c_Q|\lesssim 1$ and
$\|g_Q\|_{L^1(\eta)}\lesssim \eta(Q)$.

The first step of our arguments consists in comparing $\RR(\wt\psi\,d\LL^d)(x)$ to 
$\RR_{\ve(x)}(g\,\eta)(x)$, with $\ve(x)=\ell(Q)$ if $x\in Q\in HD_1^{k_0}$, and $\ve(x)=0$ otherwise. We will prove that, for each $Q\in HD_1^{k_0}$,
 \begin{equation}\label{eqtech4}
|\RR(\Theta(Q)|\psi_Q|\,d\LL^d)(x) - \RR_{\ve(x)}(g_Q\,\eta)(x)|\lesssim \frac{\Theta(Q)\,\ell(Q)^{s+\beta}}{\dist(x,Q)^{s+\beta}+
\ell(Q)^{s+\beta}},
\end{equation}
 where $\beta$ is some positive absolute constant. The arguments to prove this estimate
 are similar to the ones in Lemma \ref{lemdifer} or \rf{eqdah3}, although the fact that $\psi_Q$ may be
 not compactly supported introduces some additional difficulties. So we will show the details. 
 
 Suppose first that $\dist(x,Q)\leq 2\,\ell(Q)=2\,\diam(Q)$. From \eqref{equas10}, we obtain

\begin{align*}
|\RR(\Theta(Q)|\psi_Q|\,d\LL^d)(x)| &\lesssim \Theta(Q)\int_{\dist(y,Q)|\leq 4\ell(Q)} \frac{\ell(Q)^{s-d}}{|x-y|^s}\,d\LL^d(y)\\
&\quad + \Theta(Q)\int_{\dist(y,Q)|> 4\ell(Q)} \frac{\ell(Q)^{s+\alpha}}{|x-y|^s\,|y-z_Q|^{d+\alpha}}\,d\LL^d(y)\\
& \lesssim \Theta(Q),
\end{align*}
taking into account that $|x-y|\approx |y-z_Q|>4\ell(Q)$ for the estimate of the last integral.
On the other hand, we also have
$$|\RR_{\ell(Q)}(g_Q\,\eta)(x)| \leq \int_{|y-x|>2\ell(Q)} \frac{1}{|x-y|^s}\,|g_Q(y)|\,d\eta(y)
\leq \frac{1}{\ell(Q)^s}\|g_Q\|_{L^1(\eta)}\lesssim \Theta(Q).$$
Thus \rf{eqtech4} holds when  $\dist(x,Q)\leq 2\,\ell(Q)$.
 
Suppose now that $\dist(x,Q)>2\,\ell(Q)$. For such points, $\RR_{\ell(Q)}(g_Q\,\eta)(x) = 
\RR(g_Q\,\eta)(x)$. Thus we can write
\begin{align*}
|\RR(\Theta(Q)|\psi_Q|\,&d\LL^d)(x) - \RR_{\ve(x)}(g_Q\,\eta)(x)|\\ & \leq
\int K^s(x-y) \bigl[\Theta(Q)|\psi_Q(y)|\,d\LL^d(y) - g_Q(y)\,d\eta(y)\bigr] \\
& = \int_{|y-z_Q|\geq 2|x-z_Q|}\ldots + \int_{\frac12|x-z_Q|<|y-z_Q|<2|x-z_Q|}\ldots + \int_{|y-z_Q|\leq \frac12|x-z_Q|}\ldots\\& =:I_1 + I_2 + I_3.
\end{align*}

To estimate $I_1$, notice that $|y-x|\approx|y-z_Q|$ in the domain of integration. 
Also $g_Q(y)=0$ because $|y-z_Q|\geq 2|x-z_Q|> 4\ell(Q)$.
So we have
$$|K^s(x-y)|\leq \frac1{|x-y|^s}\lesssim \frac1{|y-z_Q|^s}$$
and, by the estimate \rf{equas10},
\begin{align*}
|I_1| &\lesssim \Theta(Q)\int_{|y-z_Q|\geq 2|x-z_Q|}\frac1{|y-z_Q|^s}\,|\psi_Q(y)|\,d\LL^d(y)\\
&\lesssim \Theta(Q)\int_{|y-z_Q|\geq 2|x-z_Q|}\frac1{|y-z_Q|^s}\,\frac{\ell(Q)^{s+\alpha}}{|y-z_Q|^{d+\alpha}}\,d\LL^d(y)\lesssim \frac{\Theta(Q)\,\ell(Q)^{s+\alpha}}{|x-z_Q|^{s+\alpha}}.
\end{align*}

Let us turn our attention to $I_2$. Again in the domain of integration we have $g_Q(y)=0$ because
$|y-z_Q|\geq \frac12|x-z_Q|> \ell(Q)$.
By the estimate \rf{equas10} we get
\begin{align*}
|I_2| &\lesssim \Theta(Q)\int_{\frac12|x-z_Q|<|y-z_Q|<2|x-z_Q|}
\frac1{|x-y|^s}\,\frac{\ell(Q)^{s+\alpha}}{|y-z_Q|^{d+\alpha}}\,d\LL^d(y)\\
& \lesssim \frac{\Theta(Q)\,\ell(Q)^{s+\alpha}}{|x-z_Q|^{d+\alpha}}\int_{|y-z_Q|<2|x-z_Q|}
\frac1{|x-y|^s}\,d\LL^d(y)\lesssim \frac{\Theta(Q)\,\ell(Q)^{s+\alpha}}{|x-z_Q|^{s+\alpha}}.
\end{align*}

Finally we deal with $I_3$. We write it as follows:
\begin{align*}
I_3& = \int_{|y-z_Q|\leq \frac12|x-z_Q|} \bigl[K^s(x-y) - K^s(x-z_Q)\bigr]\, \bigl[\Theta(Q)|\psi_Q(y)|\,d\LL^d(y) - g_Q(y)\,d\eta(y)\bigr]\\
&\quad + K^s(x-z_Q)\int_{|y-z_Q|\leq \frac12|x-z_Q|}\bigl[\Theta(Q)\, |\psi_Q(y)|\,d\LL^d(y) - g_Q(y)\,d\eta(y)\bigr] \\
&=: I_3^a + I_3^b.
\end{align*}
To estimate $I_3^a$ we use the smoothness of the kernel $K^s$, and then we get
\begin{align}\label{eqi3a}
I_3^a&\lesssim
\int_{|y-z_Q|\leq \frac12|x-z_Q|} \frac{|y-z_Q|^\beta}{|x-z_Q|^{s+\beta}}\, \bigl[\Theta(Q)|\psi_Q(y)|\,d\LL^d(y)+ g_Q(y)\,d\eta(y)\bigr],
\end{align}
where $0<\beta\leq1$.
Using \eqref{equas10}, we obtain

\begin{align*}
\int_{|y-z_Q|\leq \frac12|x-z_Q|} &\frac{|y-z_Q|^\beta}{|x-z_Q|^{s+\beta}}\,\Theta(Q)|\psi_Q(y)|\,d\LL^d(y)\\
&\lesssim \frac{\Theta(Q)}{|x-z_Q|^{s+\beta}}
\int_{|y-z_Q|\leq \frac12|x-z_Q|}  \frac{\ell(Q)^{s+\alpha}\,|y-z_Q|^{\beta}}{\ell(Q)^{d+\alpha} +|y-z_Q|^{d+\alpha}}\,d\LL^d(y)
\end{align*}
The last integral is bounded by
\begin{align*}
C\int_{|y-z_Q|\leq \ell(Q)}  \frac{\ell(Q)^{s+\alpha+\beta}}{\ell(Q)^{d+\alpha}}\,d\LL^d(y)
+
C\int_{|y-z_Q|> \ell(Q)}  \frac{\ell(Q)^{s+\alpha}}{|y-z_Q|^{d+\alpha-\beta}}\,d\LL^d(y).
\end{align*}
Both integrals are bounded by $C\,\ell(Q)^{s+\beta}$, assuming that $\beta<\alpha$ for the last one
(for instance, we may choose $\beta=\alpha/2$).
Thus,
$$\int_{|y-z_Q|\leq \frac12|x-z_Q|} \frac{|y-z_Q|^\beta}{|x-z_Q|^{s+\beta}}\,\Theta(Q)|\psi_Q(y)|\,d\LL^d(y)
\lesssim \frac{\Theta(Q)\,\ell(Q)^{s+\beta}}{|x-z_Q|^{s+\beta}}.$$
 
It remains now to estimate the integral involving the term $g_Q(y)\,d\eta(y)$ in \rf{eqi3a}. We have
\begin{multline*}
\int_{|y-z_Q|\leq \frac12|x-z_Q|} \frac{|y-z_Q|^\beta}{|x-z_Q|^{s+\beta}}\, g_Q(y)\,d\eta(y)\\
\lesssim
\frac1{|x-z_Q|^{s+\beta}}
\int_Q \ell(Q)^{\beta}\,g_Q(y)\,d\eta(y)\lesssim \frac{\ell(Q)^{\beta}\eta(Q)}{|x-z_Q|^{s+\beta}}
=\frac{\Theta(Q)\,\ell(Q)^{s+\beta}}{|x-z_Q|^{s+\beta}}.
\end{multline*}

Concerning $I_3^b$, since $\int g_Q\,d\eta = \Theta(Q)\int |\psi_Q|\,d\LL^d$, we deduce that
\begin{align*}
I_3^b&= -K^s(x-z_Q)\int_{|y-z_Q|>\frac12|x-z_Q|}\Theta(Q)\, |\psi_Q(y)|\,d\LL^d(y) - g_Q(y)\,d\eta(y)\\
&
= -K^s(x-z_Q)\int_{|y-z_Q|> \frac12|x-z_Q|}\Theta(Q)\, |\psi_Q(y)|\,d\LL^d(y),
\end{align*}
taking into account that $g_Q(y)=0$ for $y$ in the integrals above.
So by \eqref{equas10} we get
\begin{align*}
I_3^b &\lesssim \frac{\Theta(Q)}{|x-z_Q|^s} \int_{|y-z_Q|> \frac12|x-z_Q|}
\frac{\ell(Q)^{s+\alpha}}{|y-z_Q|^{d+\alpha}}
\,d\LL^d(y)\lesssim \frac{\Theta(Q)\,\ell(Q)^{s+\alpha}}{|x-z_Q|^{s+\alpha}}.
\end{align*}

Gathering the estimates obtained for $I_1$, $I_2$ and $I_3$, since $|x-z_Q|\approx \dist(x,Q)+\ell(Q)$
(recall that $\dist(x,Q)>2\,\ell(Q)$), we derive
$$
|\RR(\Theta(Q)|\psi_Q|\,d\LL^d)(x) - \RR_{\ve(x)}(g_Q\,\eta)(x)|\lesssim \frac{\Theta(Q)\,\ell(Q)^{s+\beta}}{\dist(x,Q)^{s+\beta}+
\ell(Q)^{s+\beta}},$$
 as claimed.
So we infer that
\begin{multline*}
\int |\RR(\wt\psi\,d\LL^d)(x) - \RR_{\ve(x)}(g\,\eta)(x)|^r\,d\eta(x)
\lesssim \int  \biggl (\sum_{Q\in HD_1^{k_0}}\frac{\Theta(Q)\,\ell(Q)^{s+\beta}}{\dist(x,Q)^{s+\beta}+
\ell(Q)^{s+\beta}}\biggr)^r\,d\eta(x).
\end{multline*}
We estimate the last integral by duality. Consider a function $h\in L^{r'}(\eta)$. Then, 
\begin{multline}\label{eqmul429}
\int  \sum_{Q\in HD_1^{k_0}}\frac{\eta(Q)\,\ell(Q)^{\beta}}{\dist(x,Q)^{s+\beta}+
\ell(Q)^{s+\beta}}\,h(x)\,d\eta(x)\\= \sum_{Q\in HD_1^{k_0}} \eta(Q)\int  \frac{\ell(Q)^{\beta}}{\dist(x,Q)^{s+\beta}+
\ell(Q)^{s+\beta}}\,h(x)\,d\eta(x)
\end{multline}
For each $Q$,
 using the fact that $\eta(\lambda Q)\leq C\,\Theta(H)\,\ell(\lambda Q)^s$ for every $\lambda\geq 1$,
 as in \rf{eqmumu2} we obtain 
$$\int  \frac{\ell(Q)^{\beta}}{\dist(x,Q)^{s+\beta}+
\ell(Q)^{s+\beta}}\,h(x)\,d\eta(x)\lesssim C\,\Theta(H)\,\inf_{y\in Q} M_\eta h(y).$$
Therefore, the left side of \rf{eqmul429} does not exceed
\begin{align*}
C\Theta(H)\sum_{Q\in HD_1^{k_0}} \eta(Q)\,\inf_{y\in Q} M_\eta h(y) & \lesssim
\Theta(H) \int_{H_{k_0}}M_\eta h(y)\,d\eta(y) \\ 
& \lesssim \Theta(H)\,\eta(H_{k_0})^{1/r}\,\|h\|_{L^{r'}(\eta)}
\lesssim \Theta(H)\,\eta(H)^{1/r}\,\|h\|_{L^{r'}(\eta)}.
\end{align*}
So we deduce that
\begin{equation}\label{eqmca2}
\int |\RR(\wt\psi\,d\LL^d)(x) - \RR_{\ve(x)}(g\,\eta)(x)|^r\,d\eta(x)\\
\lesssim \Theta(H)^r\,\eta(H) = \Theta(H)^r\,\nu(H).
\end{equation}

The estimate \rf{eqacpsi} is a consequence of \rf{eqmca2} and the fact that
$$\int |\RR_{\ve(x)}(g\,\eta)(x)|^r\,d\eta(x)\\
\lesssim \Theta(H)^r\,\eta(H).$$
This inequality is easily derived from the $L^r(\eta)$ boundedness of the operator
$\RR_{\Phi,\eta}$,
where $\Phi$ is the function defined in Lemma \ref{lemboundeta} (with $k=k_0$).
Indeed, notice that $\ve(x)\approx\Phi(x)$ for all $x\in\supp\eta$. Then, from \rf{e.compsup''},
it follows that
\begin{equation}
\label{e.compsup}
|\RR_{\ve(x)}(g\,\eta)(x)|\lesssim  |\RR_{\Phi}(g\,\eta)(x)|+ \sup_{r\geq \Phi(x)}\frac{1}{r^s}\int_{B(x,r)} |g|\,d\eta\lesssim |\RR_{\Phi}(g\,\eta)(x)| + \Theta(H).
\end{equation}
 So we have
\begin{align*}
\int |\RR_{\ve(x)}(g\,\eta)(x)|^r\,d\eta(x)& \lesssim \|\RR_{\Phi}(g\,\eta)\|_{L^r(\eta)}^r + \Theta(H)^r\,\|g\|_{L^r(\eta)}^r\\ & \lesssim
\Theta(H)^r\,\|g\|_{L^r(\eta)}^r\lesssim \Theta(H)^r\,\eta(H_{k_0})\approx \Theta(H)^r\,\nu(H).
\end{align*}
\end{proof}

\vv
\begin{proof}[\bf Proof of \rf{eqclam43}]
Let $\nu=a\,\eta$ be a minimizing measure for $F(\nu)$ as in \eqref{e.functional}, so that, in particular,
$$F(\nu)\leq 2\lambda\,\Theta(H)^r\,\eta(H).$$
We have to show that
$$|\RR\nu(x) + f(x)|^r + r \,\RR^*\bigl ((\RR\nu + f) |\RR\nu + f|^{r-2} \nu\bigr)(x) \leq C\,\lambda^{1/r'}\,\Theta(H)^r \quad \text{ on $\supp\nu$.}$$
Let $x_0\in\supp\nu$ and
consider a ball $B=B(x_0,\rho)$, with $\rho>0$ small. For $0<t<1$ we construct a competing measure $\nu_t =
a_t\,\eta$, where $a_t$ is defined as follows:
$$a_t = (1-t\,\chi_B)a +t\,\frac{\nu(B\cap H)}{\nu(H)}\,\chi_{H}\,a.$$
Notice that, for each $0<t<1$, $a_t$ is a non-negative function such that $\nu_t(H)=\nu(H)$.

Taking into account that $$\|a_t\|_{L^\infty(\eta)}\leq \|a\|_{L^\infty(\eta)} +
t\,\frac{\nu(B\cap H)}{\nu(H)},$$
we deduce that
\begin{align*}
F(\nu_t) & = \int |\RR\nu_t + f|^r\,d\nu_{t} + \lambda\,\|a_t\|_{L^\infty(\eta)} \,\Theta(H)^r\,\eta(H)\\
& \leq \int |\RR\nu_t + f|^r\,d\nu_{t} + \lambda\,\biggl (\|a\|_{L^\infty(\eta)} 
+ t\,\frac{\nu(B\cap H)}{\nu(H)} \biggr) \,\Theta(H)^r\,\eta(H) =:\wt F(\nu_t).
\end{align*}
Since $\wt F(\nu_{0}) = F(\nu) \leq F(\nu_t)\leq \wt F(\nu_t)$ for $t\geq 0$, we infer that 
\begin{equation}\label{eqvar49}
\frac1{\nu(B)}\,\frac{d}{dt}\,\wt F(\nu_t)\biggr|_{t=0} \geq 0,
\end{equation}
with the derivative taken from the right. An easy computation gives
\begin{align*}
\frac{d}{dt}\,\wt F(\nu_t)\biggr|_{t=0} & = -\int_B |\RR\nu + f|^r\,d\nu+\frac{\nu(B\cap H)}{\nu(H)}\int_{H}|\RR\nu + f|^r\,d\nu \\
& \quad +
r \int (\RR\nu + f)\, |\RR\nu + f|^{r-2} \,\RR\biggl (-\chi_B \nu + \frac{\nu(B\cap H)}{\nu(H)}
\,\chi_{H}\nu\biggr)\,
d\nu \\
&\quad +  
\lambda\nu(B\cap H) \,\Theta(H)^r.
\end{align*}
So \rf{eqvar49} is equivalent to
\begin{multline}\label{eqmult82}
 \frac1{\nu(B)} \,
\int_B |\RR\nu + f|^r\,d\nu + \frac r{\nu(B)} \,
 \int (\RR\nu + f) |\RR\nu + f|^{r-2} \,\RR(\chi_B \nu)\,
d\nu \\
\leq \frac{\nu(B\cap H)}{\nu(B)\,\nu(H)} \biggl[\int_{H}|\RR\nu + f|^r\,d\nu+ r\int (\RR\nu + f)\, |\RR\nu + f|^{r-2} \,\RR( 
\chi_{H}\nu)
d\nu +
\lambda \,\Theta(H)^r\,\nu(H)\biggr].
\end{multline}

To estimate the right side of the preceding inequality, first we use \eqref{e.admis} to notice that
\begin{equation*}
\int_{H}|\RR\nu + f|^r\,d\nu\leq F(\nu) \lesssim \lambda \,\Theta(H)^r\,\nu(H).
\end{equation*}
We  set
\begin{equation}\label{eqak4}
\left|\int (\RR\nu + f)\, |\RR\nu + f|^{r-2} \,\RR( 
\chi_{H}\nu)\,d\nu\right|\leq \left( \int |\RR\nu + f|^r \,d\nu\right)^{1/r'}  \!\!
\left(\int |\RR(\chi_{H}\nu)|^r \,d\nu\right)^{1/r}.
\end{equation}
We claim now that
\begin{equation}\label{eqcla491}
\int |\RR(\chi_{H}\nu)|^r \,d\nu \leq C\,\Theta(H)^r\,\nu(H).
\end{equation}
Assuming this to hold, \rf{eqak4} tells us that
\begin{align*}
\left|\int (\RR\nu + f)\, |\RR\nu + f|^{r-2} \,\RR( 
\chi_{H}\nu)\,d\nu\right|& \lesssim 
F(\nu)^{1/r'} \bigl (\Theta(H)^r\,\nu(H)\bigr)^{1/r}\\& \lesssim
\lambda^{1/r'}\,\Theta(H)^r\,\nu(H),
\end{align*}
recalling \eqref{e.admis} for the last inequality.
Thus the right side of \rf{eqmult82} does not exceed
$$C(\lambda^{1/r'}\,\Theta(H)^r +  \lambda\,\Theta(H)^r)\lesssim  \lambda^{1/r'}\,\Theta(H)^r,$$
assuming that $\lambda<1$ for the last estimate. 

Let us turn our attention to the left side of \rf{eqmult82} now. We rewrite it as
$$ \frac1{\nu(B)} \,
\int_B |\RR\nu + f|^r\,d\nu + \frac r{\nu(B)} \,
 \int_B \RR^*\bigl ((\RR\nu + f) |\RR\nu + f|^{r-2} \,\nu\bigr)\,
d\nu.$$
Taking into account that the functions in the integrands are continuous on $\supp(\nu)$, 
letting the radius $\rho$ of $B$ tend to $0$, it turns out that the above expression converges to
$$|\RR\nu(x_0) + f(x_0)|^r +  r\RR^*\bigl ((\RR\nu + f) |\RR\nu + f|^{r-2} \,\nu\bigr)(x_0).$$
As a consequence, we  derive
$$|\RR\nu(x_0) + f(x_0)|^r +  r\RR^*\bigl ((\RR\nu + f) |\RR\nu + f|^{r-2} \,\nu\bigr)(x_0)
\lesssim  \lambda^{1/r'}\,\Theta(H)^r,$$
as wished.

It remains to prove \rf{eqcla491}. To this end, recall that
$H = H_{k_0}\setminus H_{k_0+l}$ 
for some $l\geq1$, and consider the enlarged set
$$\wt H_{k_0+l}=\bigcup_{Q\in HD_1^{k_0+l}} (1+\tfrac1{10}c_{sep})Q,$$
and the associated
$1$-Lipschitz function $\Phi(x) = \dist(x,\R^d\setminus\wt H_{k_0+l})$. By Lemma \ref{lemboundeta}, 
$\RR_{\Phi,\eta}$ is bounded in $L^r(\eta)$, with norm not greater than $C\,\Theta(HD_1^{k_0+l})\approx
\Theta(H)$. 
 Since $\nu=a\,\eta$ with $\|a\|_{L^\infty(\eta)}\leq3$,
we infer that $\RR_{\Phi,\nu}$ is bounded in $L^r(\nu)$, also with norm not greater than $C\,\Theta(H)$.
As $\Phi$ vanishes in $\R^d\setminus\wt H_{k_0+n}$, it follows that $K_\Phi(x,y)=(x-y)/|x-y|^{s+1}$ for $y\in \supp\nu\cap H\subset\R^d\setminus \wt H_{k_0+n}$ and thus
$$\RR(\chi_{H}\nu) = \RR_\Phi(\chi_{H}\nu).$$
Therefore,
$$\|\RR(\chi_{H}\nu)\|_{L^r(\nu)}^r = \|\RR_\Phi(\chi_{H}\nu)\|_{L^r(\nu)}^r\lesssim
\Theta(H)^r \,\nu(H).$$
\end{proof}

\vv

Let us remark that,
from the conditions \eqref{e.ss} and \eqref{e.sbr}, assuming also that the constant $\delta_W$ in \eqref{e.sbr} is small enough and following the proof of \eqref{e.dechd} we deduce that
\begin{equation}\label{equa22}
\sigma(HD_1)\geq \frac12\,\sigma(\sss_{1}(\TT)).
\end{equation}
This estimate will be used below.
\vv

\begin{lemma}\label{lem22}
Let $\TT$ be a tractable tree with root $R$ and $f$ as in \rf{eqdeff}. Then there exists some subset $G\subset \bigcup_{P\in\sss_{2}(\TT)}P$ satisfying
$$\eta(G)\geq C_6(A,B,M)\,\eta(R)$$
such that, for every $x\in G$,
\begin{equation}\label{eqda59}
|\RR\eta(x) + f(x)|\geq C_7 A^{1/r}\,\bigl (p(P_x) + q(P_x,\TT)\bigr) + C_8(A,B)\,\Theta(R),
\end{equation}
where $P_x$ stands for the cube from $\sss_{2}(\TT)$ that contains $x$.
\end{lemma}

\vv
\begin{proof}
We distinguish two cases. In the first one we assume that
\begin{equation}\label{eqcase1}
\sigma(\sss_{2}(\TT))\leq A^{-1}\,\sigma(\sss_{1}(\TT)).
\end{equation}
Then from Lemma \ref{lemvar1} and \rf{equa22} we infer that
\begin{equation}\label{equari2}
\|\RR\eta + f\|_{L^2(\eta)}^2 \gtrsim\,\Theta(HD_1)^2\,\eta(HD_1)\gtrsim\sigma(HD_1) \geq \frac12\,\sigma(\sss_1(\TT))
\geq \frac A2\,\sigma(\sss_{2}(\TT)).
\end{equation}

In the second case, \rf{eqcase1} does not hold. Then, taking also into account that $\TT$ is a tractable tree, we have
\begin{equation}\label{eqcase2}
A^{-1}\,\sigma(\sss_1(\TT))\leq \sigma(\sss_{2}(\TT))\lesssim A\,\sigma(\sss_1(\TT)).
\end{equation}
For a collection of cubes $\AZ\subset\DD$ and $1<r<\infty$, we denote
$$\sigma_r(\AZ) = \sum_{P\in\AZ} \Theta(P)^r\,\mu(P),$$
so that $\sigma(\AZ)\equiv \sigma_2(\AZ)$.
Our next objective consists in showing that for some $1<r<2$, the following holds:
\begin{equation}\label{equarir}
\|\RR\eta + f\|_{L^r(\eta)}^r \gtrsim\,A\,\sigma_r(\sss_{2}(\TT)).
\end{equation}
Notice that this estimate coincides with \rf{equari2} if we take $r=2$.

To prove \rf{equarir}, notice that, by Lemma \ref{lemvar1}, choosing $1<r<2$ and 
using \eqref{e.2stop}, we get
\begin{align}\label{eqali2}
\|\RR\eta + f\|_{L^{r}(\eta)}^{r} & \gtrsim\Theta(HD_1)^{r}\,\mu(HD_1) \\
& \gtrsim \Theta(HD_1)^{r-2}\,\sigma(HD_1) 
\nonumber\\
& \gtrsim A^{-1}\,\Theta(HD_1)^{r-2}\,\sigma(HD_2)\nonumber\\
& \gtrsim A^{-1}\,\left(\frac{\Theta(HD_2)}{\Theta(HD_1)}\right)^{2-r}\,\Theta(HD_2)^r\,\mu(HD_2)\nonumber\\
& = A^{-1}\,B^{2-r}\,\Theta(HD_2)^{r}\,\mu(HD_2).\nonumber
\end{align}
We claim now that
\begin{equation}\label{eqcla52}
\sigma_r(HD_2)\approx \sigma_r(\sss_{2}(\TT)).
\end{equation}
Observe that \rf{equarir} follows from the two preceding estimates if we assume $B$ big enough so that
$$A^{-1}\,B^{2-r} \geq A.$$
From now on we assume that $r=3/2$ if \rf{eqcase1} does not hold, say, and $r=2$ otherwise.

To prove the claim \rf{eqcla52} (only in the case $r=3/2$) we will show that
$$\sigma_r(LD_1) +  \sigma_r(LD_2) + \sigma_r(BR_1)+ \sigma_r(BR_2)\leq \sigma_r(HD_2).$$
Using  the left inequality in \rf{eqcase2} and \eqref{e.is} we obtain
\begin{equation} \label{e.h2below}
\sigma(HD_2)  \gtrsim \sigma(\sss_{2}(\TT))\geq A^{-1}\sigma(\sss_1(\TT))\gtrsim A^{-2}\,\sigma(\{R\}).
\end{equation}
Therefore, taking into account that $A\leq B$,
\begin{align*}
\sigma_r(HD_2) & =\Theta(HD_2)^{r-2}\,\sigma(HD_2)\gtrsim A^{-2}\,\Theta(HD_2)^{r-2}\,\sigma(\{R\})\\
& \gtrsim A^{-2}\,B^{2(r-2)}\,\Theta(R)^{r-2}\,\sigma(\{R\}) \geq \frac1{B^{6-2r}}\,\sigma_r(\{R\}).
\end{align*}
Also, from \rf{eqld41} and the preceding estimate
\begin{align*}
\sigma_r(LD_1) +  \sigma_r(LD_2) & 
\leq \delta_0^{\frac r{s+2}}\,\Theta(R)^r\,\mu(R) + \delta_0^{\frac r{s+2}}\,B^r\,\Theta(R)^r\,\mu(R) \\
& \lesssim  \delta_0^{\frac r{s+2}}\,B^r\,\sigma_r(\{R\})\\
& \lesssim \delta_0^{\frac r{s+2}}\,B^{r+6-2r}\,\sigma_r(HD_2)
.
\end{align*}
On the other hand,
\begin{align*}
 \sigma_r(BR_1)+ \sigma_r(BR_2) & \leq \Theta(HD_1)^{2-r} \,\sigma(BR_1) + 
\Theta(HD_2)^{2-r} \sigma(BR_2) \\
&\leq \Theta(HD_1)^{2-r}\delta_W B^{2}\sigma(R)+ \Theta(HD_2)^{2-r}\delta_W B^{4}\sigma(R)\\
&\lesssim 2 B^{2(2-r)}\delta_W B^{4}\Theta(R)^{2-r}\sigma(R)\\
&\lesssim \delta_W \,B^{4(2-r)+6}\,\sigma_r(HD_2).
\end{align*}
From the last estimates, assuming that $\delta_{0}$ and $\delta_W$ are small enough, the claim \rf{eqcla52} follows, and thus the proof of \rf{equarir} in the case $r=3/2$ is finished.

Next we will use that either for $r=2$ or $r=3/2$ \rf{equarir} holds.
First, note that from the first inequality in \rf{eqali2} one obtains
\begin{align}\label{eqali3}
\|\RR\eta + f\|_{L^{r}(\eta)}^{r} & \gtrsim \Theta(HD_1)^{r-2}\,\sigma(HD_1) 
\approx \Theta(HD_1)^{r-2}\,\sigma(\sss_1(\TT))\\ &\geq \Theta(HD_1)^{r-2}\,A^{-1} \,\sigma(\{R\})
\approx C(A,B,r)\,\sigma_r(\{R\}).\nonumber
\end{align}

\vv
Denote
$$\Theta(x) = \Theta(P_x),\qquad p(x)= p(P_x),\qquad q(x) = q(P_x,\TT),$$
where $P_x$ stands for the cube from $\sss_{2}(\TT)$ that contains $x$.
From Lemma \ref{lemdifer22} and the fact that $p(x)\approx\Theta(x)$, we get
$$\int q(x)^r\,d\eta(x) \leq c_6 \int p(x)^r\,d\eta(x) \leq c_7 \int \Theta(x)^r\,d\eta(x) = c_7\,\sigma_r(\sss_{2}(\TT)).$$
Then, by \rf{equarir} and \rf{eqali3}, we obtain
\begin{equation}\label{eqsaf3}
\|\RR\eta + f\|_{L^r(\eta)}^r \geq C_3 \,A\int \bigl (p(x) + q(x) + C_4(A,B,r)\,\Theta(R)
\bigr)^r\,d\eta(x),
\end{equation}
where $C_3$ is some absolute constant. We let $G$ be the subset of the points $x\in \bigcup_{P\in\sss_{2}(\TT)}P$ such that
$$|\RR\eta(x) + f(x)|\geq \left(\frac{C_3\,A}2\right)^{1/r}\bigl (p(P_x) + q(P_x,\TT)+ 
C_6(A,B,r) \Theta(R)\bigr).$$
Note that
$$\int_{G^c}|\RR\eta + f|^r\,d\eta \leq \frac{C_5\,A}2\, \int_R \bigl (p(P_x) + q(P_x,\TT)
+ C_4(A,B) \Theta(R)\bigr)^r \,d\eta.$$
Thus, from \rf{eqsaf3} we derive
\begin{align*}
\int_{G}|\RR\eta + f|^r\,d\eta & \geq \frac{C_3\,A}2\int_R \bigl (p(P_x) + q(P_x,\TT)
+ C_4(A,B,r) \Theta(R)\bigr)^r \,d\eta\\
&\geq C_5(A,B,r)\,\Theta(R)^r\,\eta(R).
\end{align*}
On the other hand, by H\"older's inequality,
$$\int_{G}|\RR\eta + f|^r\,d\eta\leq \left(\int|\RR\eta + f|^{2r}\,d\eta\right)^{1/2} \eta(G)^{1/2}.$$
By Lemma \ref{l.t1eta}, $\RR\eta$ is bounded in $L^{2r}(\eta)$ with norm not exceeding $C_2(B,M)\Theta(R)$. 
Recalling also that $|f(x)|\lesssim p(R)\lesssim \Theta(R)$ for all $x\in R$, it follows that
$$\int|\RR\eta + f|^{2r}\,d\eta\leq C_3(B,M)\Theta(R)^{2r}\,\eta(R).$$
By combining the last three estimates we obtain
$$\Theta(R)^r\,\eta(R) \leq C(A,B,M,r) \bigl (\Theta(R)^{2r}\,\eta(R)\bigr)^{1/2}\eta(G)^{1/2},$$
which gives $\eta(G)\geq C_{6}(A,B,M,r)\eta(R)$.
\end{proof}

\vv

\begin{proof}[\bf Proof of Lemma \ref{lemtrac}]
Let $G$ be the collection of the cubes $P\in\sss_{2}(\TT)$ such that there exists some $x\in P$
satisfying 
\begin{equation}\label{eqffi2}
|\RR\eta(x)+f(x)|\geq C_7 A^{1/r}\,\bigl (p(P) + q(P,\TT)\bigr) + C_8(A,B)\,\Theta(R),
\end{equation}
where $C_7,C_8$ are the constant implicit in \rf{eqda59}. By the preceding lemma,
$$\mu\biggl (\bigcup_{P\in G} P\biggr)= \eta\biggl (\bigcup_{P\in G} P\biggr)\geq C(A,B,M)\,\eta(R)
= C(A,B,M)\,\mu(R).$$
It is easily checked that for $x\in P\in G$,
$$|\RR (\chi_{P} \,\eta)(x)|\lesssim \Theta(P)\leq p(P).$$
Further, by Lemma \ref{lemdifer}, for all $x,y\in P$,
$$|\RR (\chi_{R\setminus P} \,\mu)(x) - \RR (\chi_{R\setminus P} \,\eta)(y)|\lesssim p(P)+ q(P,\TT).$$
Then, for all $y\in P$, by \rf{eqffi2} and the preceding estimate,
\begin{align*}
\bigl|\RR (\chi_{P^c} \,\mu)(y) -  m_R(\RR\mu&)\bigr|  =
|\RR (\chi_{R\setminus P} \,\mu)(y)+f(y)|\\
& \geq C_7 A^{1/r}\,\bigl [p(P) + q(P,\TT)\bigr] + C_8(A,B)\,\Theta(R)
- C\,\bigl (p(P) + q(P,\TT)\bigr)\\
& \geq \frac{C_7}2 A^{1/r}\,p(P)  + C_8(A,B)\,\Theta(R),
\end{align*}
assuming $A$ big enough.

Recall now that by Lemma \ref{lemcomp},
$$\bigl|\RR(\chi_{ P^c}\mu)(y)- m_P(\RR\mu)\bigr|\leq C_{9} p(P).
$$
Therefore, if $A$ is big enough again,
\begin{align*}
\bigl|m_P(\RR\mu)  -  m_R(\RR\mu)\bigr|
&\geq 
\frac{C_7}2 A^{1/r}\,p(P)  + C_8(A,B)\,\Theta(R) - C_{9}p(P) \\
& \geq 
 C_8(A,B)\,\Theta(R).
\end{align*}
Since, for $x\in P\in\sss_{2}(\TT)$,
$$\sum_{Q\in\TT} D_Q(\RR \mu)(x) = m_P(\RR\mu)  -  m_R(\RR\mu),$$
we infer that
\begin{align*}
\sum_{Q\in\TT} \|D_Q(\RR \mu)\|_{L^2(\mu)}^2 & \geq
\int_{G} \Bigl|\sum_{Q\in\TT} D_Q(\RR \mu)\bigr|^{2}\,d\mu \\
& \geq C_8(A,B)^2\,\Theta(R)^2\,\mu(G)
\geq C(A,B,M)\,\Theta(R)^2\,\mu(R).
\end{align*}
\end{proof}


\section{Proof of Main Theorem}

We now finish the proof of Theorem \ref{teopri}.
We need to prove $$ \displaystyle \sigma( \mathcal  D)\lesssim \sum_{Q\in \mathcal D} \|D_Q(\RR \mu)\|_{L^2(\mu)}^2 .$$ 
First note that by stopping time arguments we can make a partition of $\DD$ into disjoint maximal trees, i.e.\ trees of type $W$, $LW$, $MDec$ and $TInc$. So we can write
$$
 \sigma( \mathcal  D)=  \sigma( W)+  \sigma( LW)+ \sigma( MDec)+ \sigma( TInc),
$$
where, abusing notation, we have identified $W$, $LW$, $MDec$ and $TInc$ with the cubes contained
in the trees of type $W$, $LW$, $MDec$ and $TInc$, respectively.

We need the following series of lemmatas.

\begin{lemma}\label{l.sigmamdec}
Let $\TT_{0}$ be the maximal tree with root $Q^0$. Then we have
\begin{equation} \label{e.sigmadecro}
\sigma(MDec\setminus \TT_{0}) \leq C(A,B,M,\delta_{0})( \sigma( W)+ \sigma( LW)+ \sigma( TInc)) \quad \text{if  }  \TT_{0}\in MDec,
\end{equation}
and
\begin{equation} \label{e.sigmadec}
\sigma(MDec) \leq C(A,B,M,\delta_{0})( \sigma( W)+ \sigma( LW)+ \sigma( TInc)) \quad \text{if  }  \TT_{0}\notin MDec.
\end{equation}
\end{lemma}

Recall that $Q^0$ is the largest cube from $\DD$. That is, it is the starting cube in the construction of $E$.

\begin{proof}
We will prove \eqref{e.sigmadecro}, as \eqref{e.sigmadec} follows by an almost identical argument.  Let $\TT \in MDec$ and let $R=\roo(\TT)$. We first note that
\begin{equation}\label{e.estimatedec}
\sigma(\TT)\leq C(B,A) \sigma(R)
\end{equation}
To prove this we define $\mathcal S_{0}(\TT):=\{R\}$, and in general $\mathcal S_{k}(\TT)$
is the collection of the cubes from $\ttt\cap \,\TT\setminus \bigcup_{j=0}^{k-1} \mathcal S_{j}$
which are maximal.
Since $\TT \in MDec$, then $\tree(R')$ is $D\sigma$ for all $R'\in \ttt\cap\, \TT$. Therefore,
$\sigma(\sss(R'))\leq A^{-1} \sigma(R)$ and also $\sigma(\mathcal S_{k}(\TT))\leq A^{-1} \sigma( \mathcal S_{k-1}(\TT))$. Iterating this estimate we get
$$
\sigma(\mathcal S_{k}(\TT))\leq A^{-k} \sigma(R).
$$
And now the proof of \eqref{e.estimatedec} follows easily using Lemma \ref{lemdif1}:
\begin{align*}
\sigma(\TT)&\leq c'(B,\delta_0,M) \sum_{R'\in \ttt\cap \TT} \sigma(R') \leq c'(B,\delta_0,M) \sum_{k=0}^{\infty} \sigma(\mathcal S_{k}(\TT))\\
&\leq c'(B,\delta_0,M) \sum_{k=0}^{\infty}A^{-k} \sigma(R)\leq c'(B,\delta_0,M,A) \sigma(R).
\end{align*}

Next we observe that since $R=\roo(\TT)\neq Q_{0}$ and $\TT\in MDec$, then there exists a tree $\TT'\in W\cup LW\cup TInc$  such that $R\in \sss(\TT')$. We use this fact together with \eqref{e.estimatedec} to finish the proof of the lemma.

\begin{align*}
\sigma(MDec\setminus \TT_{0})  &= \sum_{\TT\in MDec\setminus \TT_{0}}\sigma(\TT)\\
 &\leq c'(B,\delta_0,M,A) \sum_{\substack{R=\roo(\TT)\\ \TT\in MDec\setminus \TT_{0}}}\sigma(R)\\
 &\leq C(B,\delta_0,M,A) \sum_{\TT'\in W\cup LW\cup TInc} \sigma(\TT').
 \end{align*}
The second inequality follows from the fact that any cube $Q\in\DD$ satisfies
 $\Theta(Q)\lesssim\Theta(Q^1)$, since the $\ell(Q)\approx\ell( Q^1)$ where $Q^1$ is the parent cube of $Q$.

\end{proof}

\begin{lemma} \label{l.rodec}
Let $\TT_{0}$ be the maximal tree with $\roo(\TT_{0})= Q^{0}$. If  $\TT_{0}\in MDec$, then
\begin{equation*} 
\sigma(\TT_{0})\lesssim \| \RR \mu \|_{L^2(\mu)}^2.
\end{equation*}
\end{lemma}

The proof of the Lemma \ref{l.rodec} will require the following result, whose proof we postpone.

\begin{lemma}\label{l.twosepcubes}
Let $Q^{0}$ be a cube that contains the support of measure $\mu$. Suppose there exist cubes $Q_{u}, Q_{d} \subset Q^0$ such that $\dist(Q_{u}, Q_{d} )\approx \ell(Q^0 )$ and $\mu(Q_{u})\approx \mu(Q_{d})\approx \mu(Q^0 )$. Then 
$$\|\RR\mu\|_{L^2(\mu)}^2\geq C\,\sigma(Q^0).$$
 \end{lemma}
 
\begin{proof}[Proof of Lemma \ref{l.rodec}]

Applying Lemma \ref{lemtouch} to $R':=Q^0$, we deduce that there exist $\delta$ and $\eta$ such that
$$
\mu\biggl(\bigcup_{Q\in \SSS(Q^0 )} Q\biggr)\geq \eta \mu(Q^0 ),
$$
where $\SSS(Q^0)= \{Q\in \sss_{0}(Q^0): \,\, \ell(Q)\geq \delta \ell(Q^0 ) \}$. Let $Q_{u}\in \SSS(Q^0 )$ with maximal $\mu$ measure, then
\begin{equation}\label{e.mass}
\mu(Q_{u})\gtrsim \delta^{d}\eta \, \mu(Q^0 ).
\end{equation}

We now distinguish three possible cases.
In the first one, let us suppose $Q_{u} \in BR(Q^0 )$. Then part (b) of Lemma \ref{lemt0}  and \eqref{e.mass} gives
\begin{align*}
 \| D_{Q^0 }\RR \mu \|_{L^2(\mu)} \mu(Q^0 )^{1/2} & \geq \int_{Q^0 }| D_{Q^0 }\RR \mu |d\mu \\
 										& \geq \int_{Q_{u}}|m_{Q_{u}}\RR \mu- m_{Q^0 }\RR \mu |d\mu \\
										& \geq \frac{M}{2}\Theta(Q^0 ) \mu(Q_{u}) 
										 \geq \frac{M}{2} \delta^{d}\eta \, \Theta(Q^0 ) \mu(Q^0 ).
\end{align*}
We use this estimate to get the desired result:
$$
 \| \RR \mu \|_{L^2(\mu)}^{2} \geq  \| D_{Q^0 }\RR \mu \|_{L^2(\mu)}^{2}\geq \frac{M^{2}}{4} \delta^{2d}\eta^{2}\sigma(Q^0 ).
 $$
 
In the second case, let us assume $Q_{u}\in HD(Q^0 )$. We claim that $\mu(Q_{u})\leq \frac{1}{2}\mu(Q^0 )$.  Suppose that is not the case, then 
$$
\sigma(\sss(Q^0 ))\geq \sigma(Q_{u})> \frac{B^{2}}{2}\Theta(Q^0 )^{2}\mu(Q^0 ).
$$
This contradicts the fact that the tree is $\sigma$-decreasing. Therefore $\mu(Q_{u})\leq \frac{1}{2}\mu(Q^0 )$ and obviously $\mu(Q^0 \setminus Q_{u})\geq \frac{1}{2}\mu(Q^0 )$. Consider the cube $Q_{d}$ such that $Q_{d}\subset Q^0 \setminus Q_{u}$,  $\ell(Q_{d})=\ell(Q_{u})$ and $\mu(Q_{d})$ is maximal. Then $\mu(Q_{d})\approx \mu(Q^0 )$ with constant depending on $\delta$ and by the separation condition \eqref{eqsepara} $\dist(Q_{u}, Q_{d}) \approx \ell(Q_{u}) \approx \ell(Q^0 )$ with constant also depending on $\delta$. We see that the cubes $Q_{u},Q_{d}$ verify the assumptions of Lemma \ref{l.twosepcubes}, therefore we conclude
$$
 \| \RR \mu \|_{L^2(\mu)}^2\geq C\sigma(Q^0 ).
 $$
 
 The third case is very similar to the previous one. Let us now assume that $Q_{u}\in LD(Q^0 )$, then $\Theta(Q_{u})\leq \delta_{0} \Theta(Q^0 )$, which in particular gives that  $\mu(Q_{u})\leq \delta_{0} \mu(Q^0 )$. We now proceed as in the second case to obtain the desired estimate.
\end{proof}
\vv

\begin{proof}[Proof of Lemma \ref{l.twosepcubes}]  
By replacing $Q_u$ and $Q_d$ by suitable descendants, we may assume that
\begin{equation}\label{eqdaqo39}
\ell(Q^0)\approx\dist(Q_u,Q_d)>2\bigl(\ell(Q_u) + \ell(Q_d)\bigr).
\end{equation}
Let $L_{1}$ be the shortest segment that joins the cubes $Q_{u}$ and $Q_{d}$ and let us call $u$ a unit vector parallel to the segment. Let  $\mathcal H_{2}$ be the hyperplane that is perpendicular to $L_{1}$ and passes through the middle point of $L_{1}$. Then $\mathcal H_{2}$ divides $Q^0 $ in two regions, $D$ and $U$, so that by \rf{eqdaqo39} $D$ contains $Q_d$ and $U$ contains $Q_u$, say.  
Let us denote by $\RR^u$ the singular integral operator associated with the kernel $K^s(x-y)\cdot u$. By
the antisymmetry of the kernel, we have
\begin{align*}
 \| \RR \mu \|_{L^2(\mu)}\, \mu(D)^{1/2} &   \geq 
 \left|\int_{D}\RR^{u}\mu\,d\mu\right| =
 \left|\int_{D}\RR^{u}(\chi_{U}\mu)d\mu\right|\\
								&=\int_{Q_{u}}\int_{Q_{d}} 
								\frac{|\langle x-y, u\rangle|}{|x-y|^{s+1}}d\mu(x)d\mu(y)\\
								&\gtrsim \int_{Q_{u}}\int_{Q_{d}} 
								\frac{\dist(Q_{u},Q_{d})}{\ell(Q^0 )^{s+1}}d\mu(x)d\mu(y)\\
								&= \frac{\dist(Q_{u},Q_{d})}{\ell(Q^0 )^{s+1}}\mu(Q_{u})\mu(Q_{d})
								\approx \frac{\mu(Q^0 )^{2}}{\ell(Q^0 )^{s}}.
\end{align*}
 From the above estimates the lemma follows.
\end{proof}

\vv
Theorem \ref{teopri} is an immediate consequence of the preceding results. 
Indeed, by Lemmas \ref{l.sigmamdec} and \ref{l.rodec}
we have
$$\sigma(MDec) \leq C(A,B,M,\delta_{0})\left( \sigma( W)+ \sigma( LW)+ \sigma( TInc)
+ \|\RR \mu \|_{L^2(\mu)}^2
\right).$$
Thus,
$$\sigma(\DD)
\leq 
C(A,B,M,\delta_{0})\left( \sigma( W)+ \sigma( LW)+ \sigma( TInc)
+ \|\RR \mu \|_{L^2(\mu)}^2
\right).$$
By Lemmas \ref{lemtw} and \ref{lemtame} we know that 
\begin{align*}
\sigma( W)+ \sigma( LW)+ \sigma( TInc)& \leq C(A,B,\delta_0,M,\delta_W)\, \sum_{Q\in\DD}\|D_Q(\RR \mu)\|_{L^2(\mu)}^2 \\& =C(A,B,\delta_0,M,\delta_W)\,\|\RR \mu \|_{L^2(\mu)}^2,
\end{align*}
and so we are done.

\enlargethispage{1cm}

\end{document}